\documentclass[11pt,english]{article}
\usepackage{tikz}
\usetikzlibrary{decorations.pathreplacing}
\usepackage[ruled,vlined]{algorithm2e}

\usepackage{amsmath,amssymb,amsthm,mathtools}
\usepackage{enumerate}
\usepackage[margin=1.1cm]{caption}

\usepackage[hyphens]{url}
\usepackage{hyperref}
\hypersetup{colorlinks=true,linkcolor=blue,citecolor=blue,pdfpagemode=UseNone,pdfstartview={XYZ null null 1.00}}
\hypersetup{breaklinks=true}

\usepackage{amsmath, amsthm, amsopn, amssymb}
\usepackage{xcolor}
\usepackage{graphicx}

\usepackage[margin = 22mm, bottom=21mm,footskip=8mm,top=21mm]{geometry}
\theoremstyle{plain}
\newtheorem*{theorem*}{Theorem}
\newtheorem{theorem}{Theorem}[section]

\newtheorem{proposition}[theorem]{Proposition}
\newtheorem*{claim*}{Claim}

\newtheorem{conjecture}[theorem]{Conjecture}
\newtheorem{problem}[theorem]{Problem}
\newtheorem{question}[theorem]{Question}
\theoremstyle{remark}

\newcommand{\F}{\mathcal{F}}

\let\emptyset\varnothing

\let\originalleft\left
\let\originalright\right
\renewcommand{\left}{\mathopen{}\mathclose\bgroup\originalleft}
\renewcommand{\right}{\aftergroup\egroup\originalright}
\date{}

\title{Constructions in combinatorics via neural networks}

\author{
Adam Zsolt Wagner\thanks{School of Mathematical Sciences, Tel Aviv University, Tel Aviv, Israel. Email: \texttt{adamwagner@mail.tau.ac.il}}}

\begin{document}
\maketitle

\begin{abstract}
We demonstrate how by using a reinforcement learning algorithm, the deep cross-entropy method, one can find explicit constructions and counterexamples to several open conjectures in extremal combinatorics and graph theory. Amongst the conjectures we refute are a question of Brualdi and Cao about maximizing permanents of pattern avoiding matrices, and several problems related to the adjacency and distance eigenvalues of graphs. 
\end{abstract}

\section{Introduction}

Computer-assisted proofs have a long history in mathematics, including breakthrough results such as the proof of the four color theorem in 1976 by Appel and Haken~\cite{4ct}, and the proof of the Kepler conjecture in 1998 by Hales~\cite{kepler}. Recently, significant progress has been made in the area of machine learning algorithms, and they have have quickly become some of the most exciting tools in a scientist's toolbox. In particular, recent advances in the field of reinforcement learning have led computers to reach superhuman level play in Atari games~\cite{atari} and Go~\cite{go}, purely through self-play.

With these recent advances in mind, it is natural to wonder if there is any way reinforcement learning algorithms could be used to find explicit counterexamples to conjectures in combinatorics and graph theory, if we do not give the algorithm any prior knowledge about the problems. Our aim is to demonstrate that this approach could potentially be fruitful: while we did not succeed in refuting any of the most famous conjectures in the field, we will present counterexamples and constructions to a wide variety of lesser known open problems. 

In all applications, the agent starts the learning process with zero prior knowledge about the problem itself. It generates a construction and receives a feedback from a scoring function on how well it did. Repeating this process many times, the agent eventually learns how to get a better score in this `game'. If this score is larger than the score of the conjectured best construction, then we have succeeded in finding a counterexample to the conjecture. An advantage of using reinforcement learning algorithms in this manner is that we can use what is essentially the same exact program to try and attack all mathematical conjectures which might have a finite counterexample -- the only thing we need to change in the code is the function that calculates the score of a given construction.

There are several natural candidate reinforcement learning algorithms for this task. In this work we are primarily focused on conjectures in graph theory, and as an $n$-vertex graph can be represented by $\frac{n(n-1)}{2}$ binary decisions, one for each edge, we will focus on algorithms that are well-suited for problems with small action spaces. One may initially assume that the most natural choice would be to use Deep Q-Networks and their variants, such as Double Deep Q-Networks~\cite{double} and Dueling Deep Q-Networks~\cite{dueling}, which have found a lot of success in recent years. However, these methods take a long time to train in our sparse reward setting, where we only give the agent feedback at the very end of each session (so roughly at every $\frac{n^2}{2}$ steps if we generate a graph edge-by-edge). While there are several  ways to potentially resolve this issue, such as giving some kind of artificial reward during the sessions to guide the agent, doing so would introduce its own set of problems and would defeat our goal of refuting conjectures without prior knowledge about the problem.

With our limited computing resources, we have found significantly more success with the algorithm called the \emph{deep cross-entropy method}. While this algorithm is not as famous than the above-mentioned Deep Q-Networks, it has good convergence and appeared to be much less sensitive to choosing the right hyperparameters. We will present some examples where this approach has resulted in counterexamples and constructions that would have been difficult to find by hand or with other programming methods. At the end of the paper we also present two counterexamples that we found by applying LP solvers to conjectures that can be formulated as linear programs.

The paper is organized as follows.
\begin{itemize}
\item In Section~\ref{sec:neural} we present our results we obtained via the deep cross-entropy method.
\begin{itemize}
    \item In~\S\ref{subsec:crossentropy} we give a short introduction to the cross-entropy method and describe how we will use it to produce constructions to extremal combinatorics problems.
    \item In~\S\ref{subsec:matchingeigen} we illustrate the method by finding  counterexamples to a conjecture about the sum of the largest eigenvalue and matching number of graphs, which was proposed in~\cite{aouch}. 
    \item In~\S\ref{subsec:aouch2} we refute a similar conjecture of Aouchiche--Hansen~\cite{aouchhansen} about the distance spectrum and proximity of graphs. 
    \item In~\S\ref{subsec:peaksfar} we refute an old conjecture of Collins~\cite{collins} by showing that the peaks of the coefficient sequences of the adjacency and distance polynomials of trees can be far apart.
    \item In~\S\ref{subsec:transmission} we show that transmission regularity of graphs is not preserved under cospectrality of the distance Laplacian, answering a question of Hogben and Reinhart~\cite{surveydistance}.
    \item In~\S\ref{subsec:perm312} we address a problem of Brualdi and Cao~\cite{brualdi} about maximizing the permanent of an $n\times n$, 312-pattern avoiding binary matrix. Among others, we find that the best possible answers for $n\leq 8$ are given by the rather remarkable sequence
    $$1,\quad 2,\quad 4,\quad 8,\quad 16,\quad 32,\quad 64,\quad 120.$$
\end{itemize}
\item In Section~\ref{sec:lp} we present two constructions obtained via LP solvers.
\begin{itemize}
    \item In~\S\ref{subsec:carla} we refute a conjecture of Aaronson–Groenland–Grzesik–Kielak–Johnston~\cite{aaronson} about a problem of covering certain subsets of the hypercube with few hyperplanes.
    \item In~\S\ref{subsec:domotor} we answer a problem of Kir\'aly--Nagy--P\'alv\"olgyi--Visontai~\cite{domotor} about the maximum size of weakly cross-intersecting set-pair systems.
\end{itemize}
\end{itemize}

\section{Reinforcement learning and the cross-entropy method}\label{sec:neural}

Reinforcement learning methods have a long history of being used to tackle complex combinatorial optimization problems. In recent years this area has seen a myriad of new developments and scientific papers appearing. Often new results are concerned with the following question: how is it possible to tackle NP-hard problems, such as the traveling salesman problem, in practical time? For a general overview we refer the reader to the numerous  surveys written in the past three years, such as~\cite{nnsurvey1,nnsurvey2,nnsurvey3,nnsurvey4}. 

In the present paper we take a  different approach. We will use reinforcement learning methods to find explicit counterexamples to some conjectures in combinatorics and graph theory. In this section we will present five examples where this approach was successful. The common theme in these examples that we chose to present is that the constructions found by the neural network are  non-trivial, and it is not clear how one could have found these by hand or by relying only on more traditional programming methods.

There are a plethora of algorithms within reinforcement learning, here we will focus on the so-called \emph{deep cross-entropy method}, which we will use in all five applications to find the constructions. We will use a deep neural network to approximate the policy, which determines what action we should take in a given state. With the cross-entropy method, the neural network learns only to predict which move is best in a given state, and does not explicitly learn a value function for the states or state-action pairs. Given any state as an input to the neural net, the output is a probability distribution on all the possible moves in that state, with higher probability assigned to the moves that the agent thinks are best.

In this work we will only give a basic description of the cross-entropy method and focus on how we applied this method in practice to find the constructions. For a basic introduction to neural networks, and a thorough description of the theoretical background of the cross-entropy method and of various practical considerations when applying it, we refer the reader to the recent book~\cite[Chapter 4]{lapan2020deep}.

\subsection{Applying the cross-entropy method to problems in extremal combinatorics}\label{subsec:crossentropy}

Our very first step in attacking an extremal problem is to find a good way to encode constructions as words. Given a combinatorial problem, we can often easily translate it to a problem about generating a word of certain length from a finite alphabet. For example, if our task is to find a large antichain in the Boolean lattice $\{0,1\}^n$ then the answer can be represented by a 0-1 string of length $2^n$, where the $i$-th bit is one precisely if the $i$-th element in some linear ordering of $\{0,1\}^n$ is included in the antichain. Similarly, generating $n$-vertex graphs is equivalent to generating 0-1 sequences of length $\frac{n(n-1)}{2}$; trees on $n$ vertices can be represented by their Pr\"ufer codes as words of length $n-2$ on an alphabet of size $n$, and so on.

We will generate constructions with the neural network as follows. We first ask it to predict what the best first letter should be. The output is a probability distribution on the alphabet, from which we can sample an element randomly and feed it back into the network to ask what the best second letter is. In general, having generated the letters $a_1,a_2,\ldots,a_{k-1}$, we can feed this partial word on $k-1$ letters into the network to get a probability distribution on the best next letter to use, conditioned on our previous $k-1$ decisions, and sample from it randomly to obtain $a_k$.

\begin{figure}[hbt]
    \centering
    \includegraphics{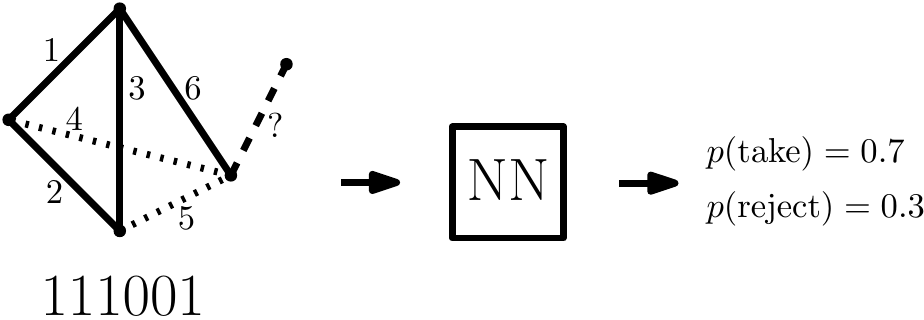}
    \caption{We are generating a graph, and so far the network has included edges 1,2,3,6, and rejected edges 4,5 (dotted), corresponding to the sequence 111001. We input this into the neural network to get a probability distribution on whether to include or reject the next edge.}
    \label{fig:nn_illustration}
\end{figure}

~

\begin{algorithm}[H]
\SetAlgoLined
 Initialize a neural network\;
 \While{the best construction found is not a counterexample}{
\For{$i\gets1$ \KwTo $N$}{
 $w\gets $ empty string\;
 \While{not terminal}{
 Input $w$ into the neural net to get a probability distribution $F$ on the next letter\;
 Sample next letter $x$ according to $F$\;
 $w\gets w+x$\; 
 }
}
  Evaluate the score of each construction\;
  Sort the constructions according to their score\;
  Throw away all but the top $y$ percentage of the constructions\;
  \For{all remaining constructions}{
    \For{all (observation, issued action) pairs in the construction }{
         Adjust the weights of the neural net slightly to minimize the cross-entropy loss between issued action and the corresponding predicted action probability\;
        }
    }
  Keep the top $x$ percentage of constructions for the next iteration, throw away the rest\;
 }
 \caption{The deep cross-entropy method}
\end{algorithm}

The network receives feedback only when a session finishes, i.e.~when we have created the entire word. The feedback is given by a reward function that the neural network has no knowledge about and treats as a black box. This reward function is unique to every problem. When trying to create large antichains the reward might be the number of elements in the family minus the number of comparable pairs in it; when aiming to create a graph with a large third eigenvalue it might simply be the third eigenvalue, etc. This reward at the end of each session is the only feedback the network receives: at no point does it have any knowledge on what problem it is trying to solve and how the reward is calculated. 

In every iteration we generate a large number of random sessions (constructions) following the above method. We calculate the reward for each, and throw away all but the top $y$ percentile. Next we let the neural network learn from the remaining sessions, meaning that we adjust the weights of the neural network a little bit to make it more likely to output moves that were used in the best performing sessions. The idea of this is to reinforce our neural network to carry out those actions that have led to good rewards.  

Note that in most implementations of the cross-entropy method found online, at the end of each iteration either all sessions are discarded, or the top few sessions are kept alive for a few more iterations. In some cases we had better success by letting the top few sessions survive indefinitely, until they are outperformed by other sessions. This change was implicitly suggested in~\cite{tds}, where they recommend a similar modification when the number of successful sessions is small.

The implementation of the input is as follows. When generating $n$-vertex graphs the input is two vectors of length $\frac{n(n-1)}{2}$. The first vector has a 1 in all positions that correspond to edges that the agent has decided to take, and it has a 0 in all the places that the agent has rejected, or not considered yet. The second vector has a 0 in every position, except for the one position that corresponds to the edge that is currently being considered, where it has a 1. The purpose of the second vector is so that the agent can figure out which edge it is currently making a verdict on. In all applications where we generate a binary string, the neural network has three hidden layers with 128, 64 and 4 neurons respectively. The learning rate is chosen experimentally to strike a balance between the speed of convergence and avoiding getting stuck in suboptimal constructions.

This representation of the input as two one-dimensional vectors and the use of a sequential network of dense layers makes this a simple, multi-purpose setup to use. For specific problems, for example when the task is to construct graphs with certain properties, other input formats and networks architecture such as graph neural networks could be considered. See~\cite{deepmind} and the book~\cite{gnnbook} for an overview of the recent applications of graph neural networks for combinatorial optimization problems.
For the related problem of having to generate a graph that looks similar to a fixed target collection of graphs, other generative models are discussed in the survey~\cite{graphgensurvey}. In particular, Generative Adversarial Networks were used in~\cite{gan1} where the graphs were constructed via random walks, and in~\cite{gan2} where the authors predict the adjacency matrix at once, and utilize a permutation-invariant discriminator for the vertex set.

The full Python code used in Sections~\ref{subsec:matchingeigen} and~\ref{subsec:aouch2} is available here:~\cite{code}.

\subsection{Example: the sum of matching number and largest eigenvalue of graphs}\label{subsec:matchingeigen}

We will use the following conjecture appearing in~\cite{aouch} to illustrate our method. 

\begin{conjecture}[\cite{aouch}\label{conj:aouch}]
Let $G$ be a connected graph on $n\geq 3$ vertices, with largest eigenvalue $\lambda_1$ and matching number $\mu$. Then
$$\lambda_1+\mu \geq \sqrt{n-1} +1.$$
\end{conjecture}

Conjecture~\ref{conj:aouch} was originally obtained using AutoGraphiX~\cite{autographix}, a software that can be used to find conjectures about relations between various graph parameters in an automated way. After we have found a counterexample with our method, we have learned that Conjecture~\ref{conj:aouch} has already been disproved by Stevanovi\'c~\cite{stevanovic}. Nevertheless, we can use this conjecture to demonstrate how simple it can be to find constructions with the cross-entropy method. The counterexample given in~\cite{stevanovic} has $n=600$ vertices, but we can use our methods to find a smaller explicit counterexample.

We used the cross-entropy method with $n=19$, where the reward function we use is simply $\lambda_1+\mu$, which we will try to minimize. Figure~\ref{fig:aouch_evolve} shows how the average reward of the top 10\% of sessions of each iteration evolves over time. 

\begin{figure}[hbt]
    \centering
    \includegraphics[scale=0.8]{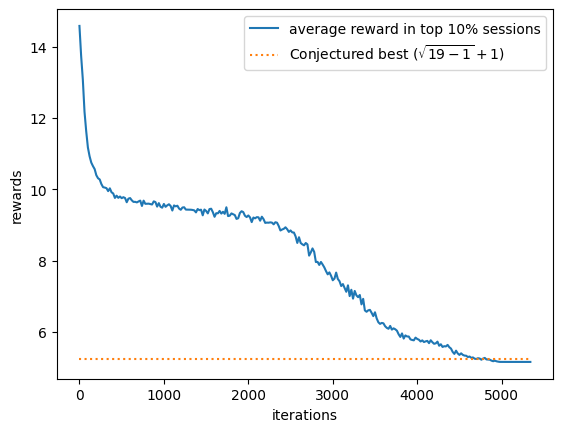}
    \caption{The average value of $\lambda_1+\mu$ of the best 10\% of the sessions in each iteration decreases over time. After 5000 iterations it goes slightly below the conjectured threshold of $\sqrt{19-1}+1$, meaning we have found a counterexample to Conjecture~\ref{conj:aouch}.}
    \label{fig:aouch_evolve}
\end{figure}

\begin{figure}[hbt]
    \centering
   
    \includegraphics[scale=0.2]{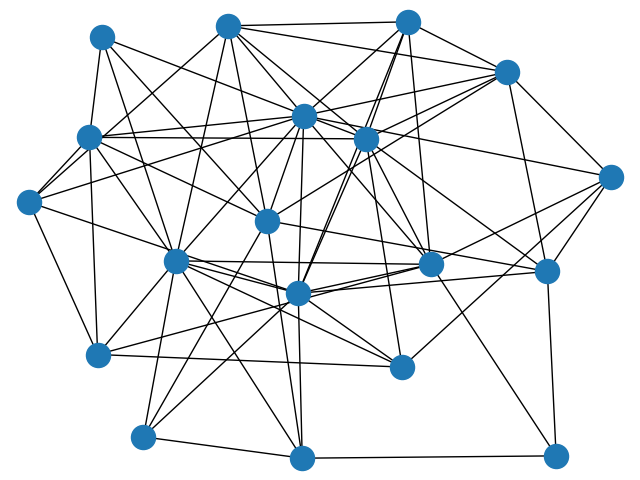}
    \includegraphics[scale=0.2]{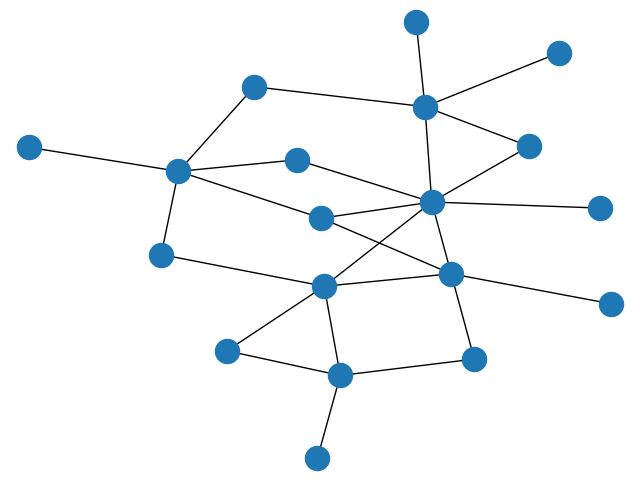}
    \includegraphics[scale=0.2]{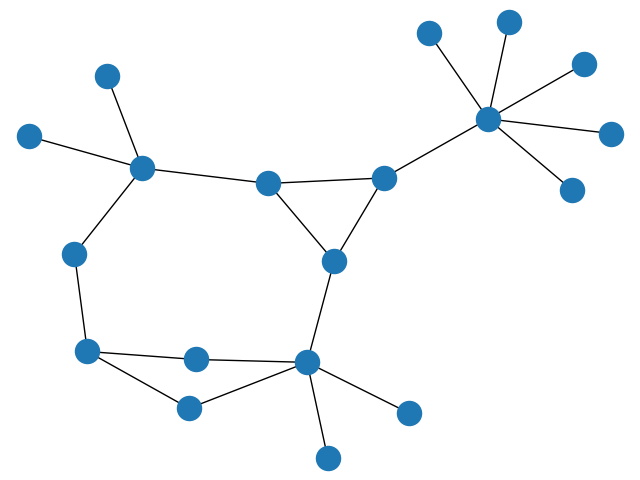}
    \includegraphics[scale=0.2]{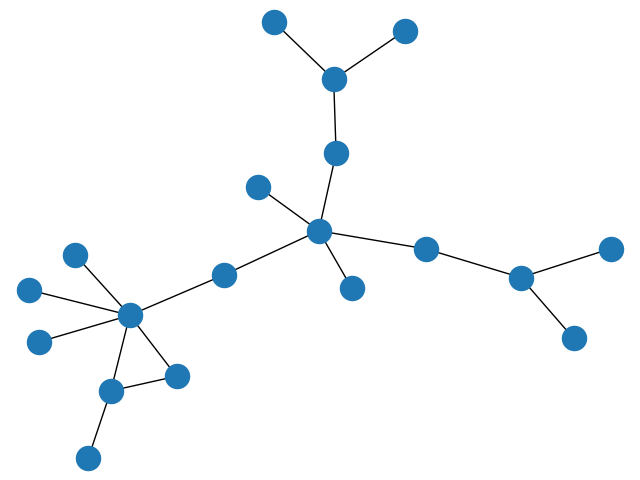}
    \includegraphics[scale=0.2]{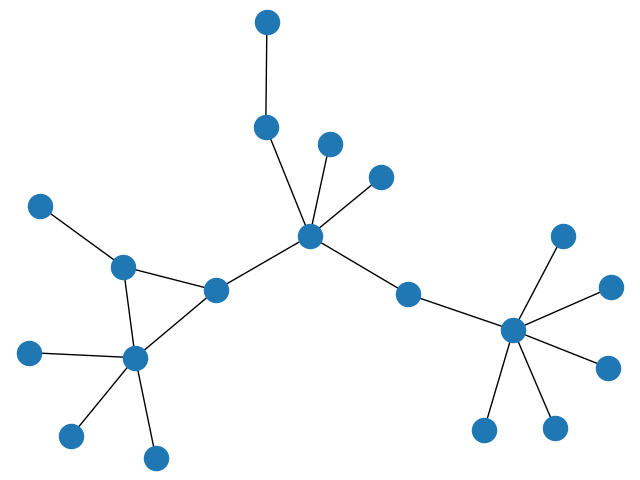}
    
    ~
    
    \includegraphics[scale=0.2]{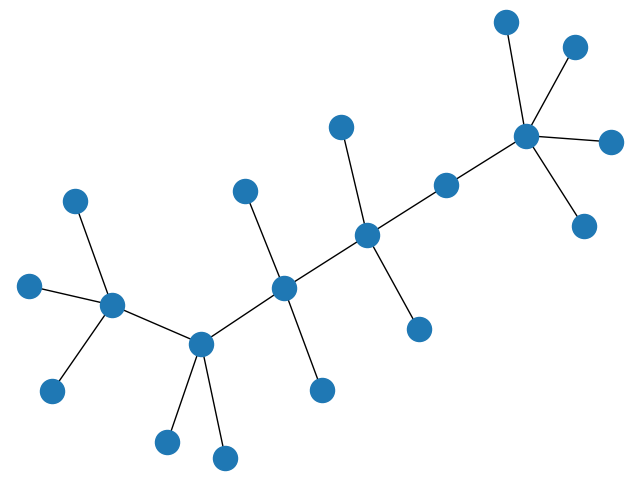}
    \includegraphics[scale=0.2]{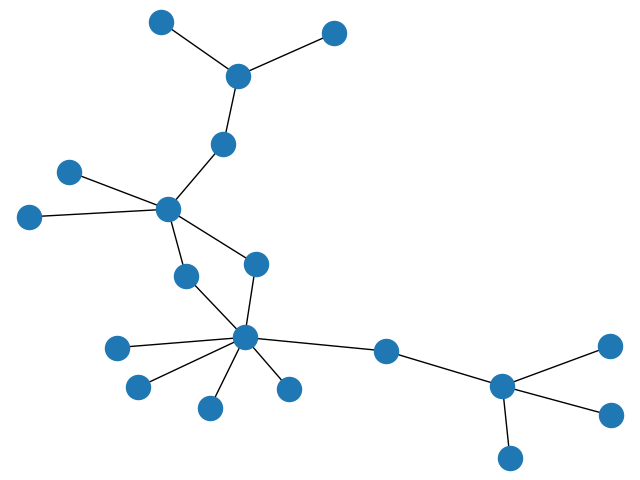}
    \includegraphics[scale=0.2]{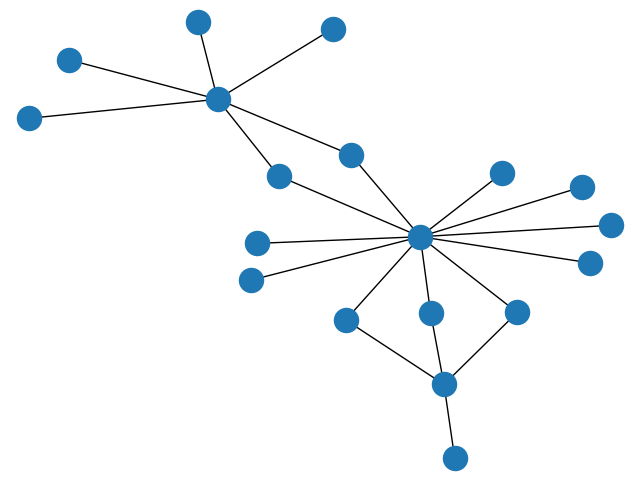}
    \includegraphics[scale=0.2]{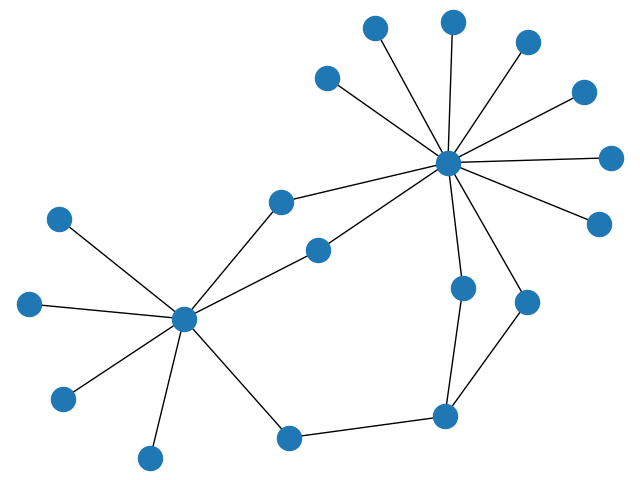}
    \includegraphics[scale=0.2]{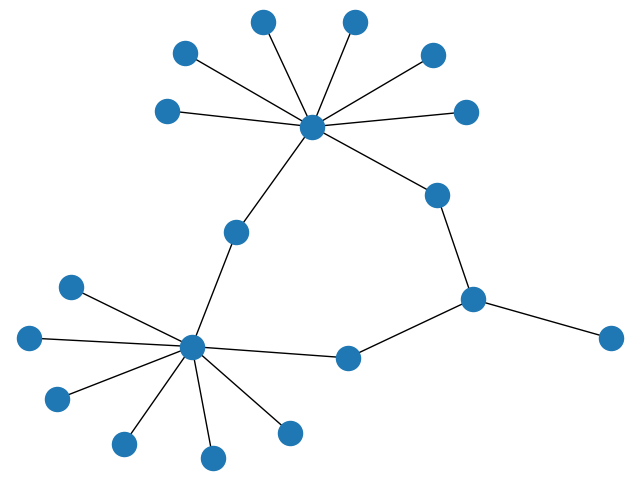}
    
    ~
    
    \includegraphics[scale=0.2]{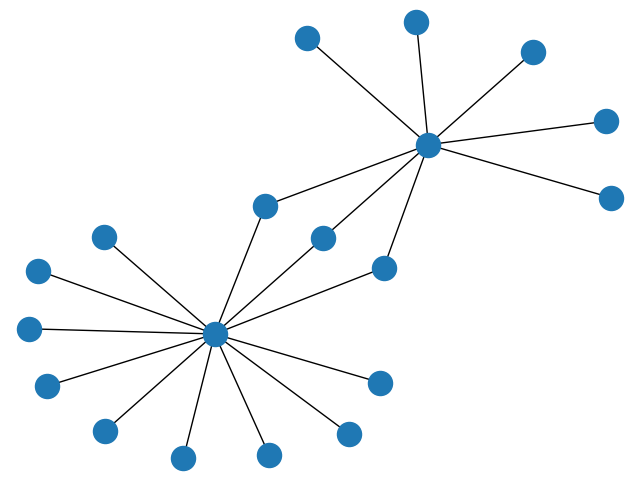}
    \includegraphics[scale=0.2]{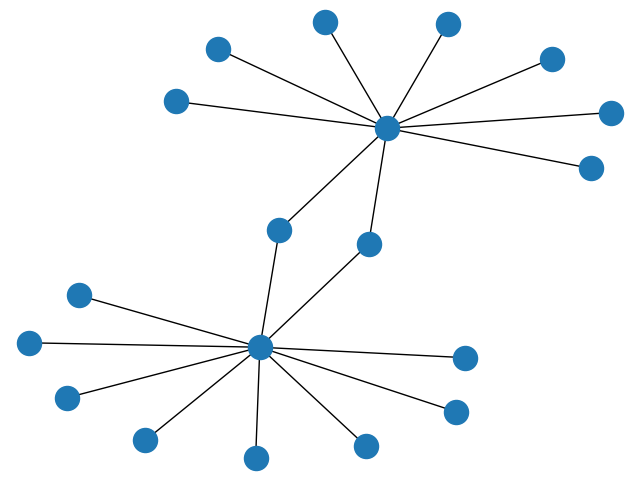}
    \includegraphics[scale=0.2]{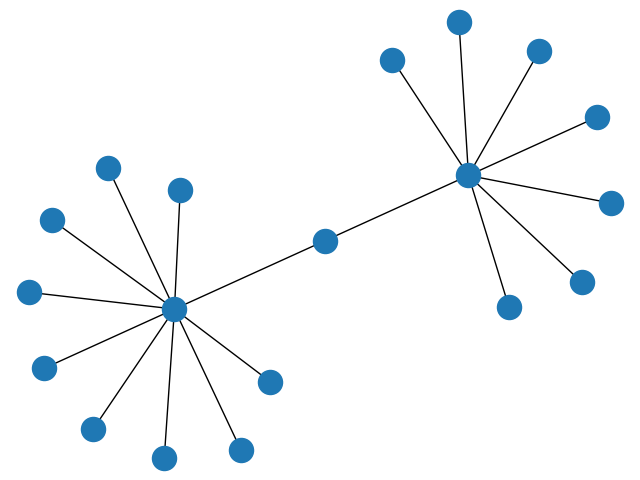}
    \includegraphics[scale=0.2]{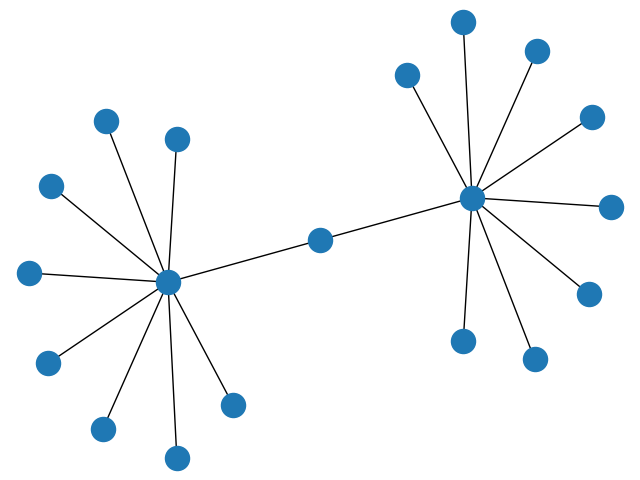}

    \caption{The evolution of the best construction over time. The network quickly realizes that sparse graphs are best, and eventually the ``balanced double star'' structure emerges.}
    \label{fig:aouch_timeline}
\end{figure}

What is remarkable about Conjecture~\ref{conj:aouch} is that it appears to be true for $n\leq 18$ -- the smallest counterexample we could find is for $n=19$, depicted in Figure~\ref{fig:aouch}. It has largest eigenvalue $\sqrt{10}$ and matching number 2, so $\lambda_1+\mu\approx 5.16 < 5.24\approx \sqrt{19-1}+1. $ Figure~\ref{fig:aouch_timeline} shows how the best sessions evolved as the number of iterations increased. While there is a significant run-to-run variation, it typically takes a few hours on an average PC for the program to find the counterexample.

\begin{figure}[hbt]
    \centering
    \includegraphics[scale=0.45]{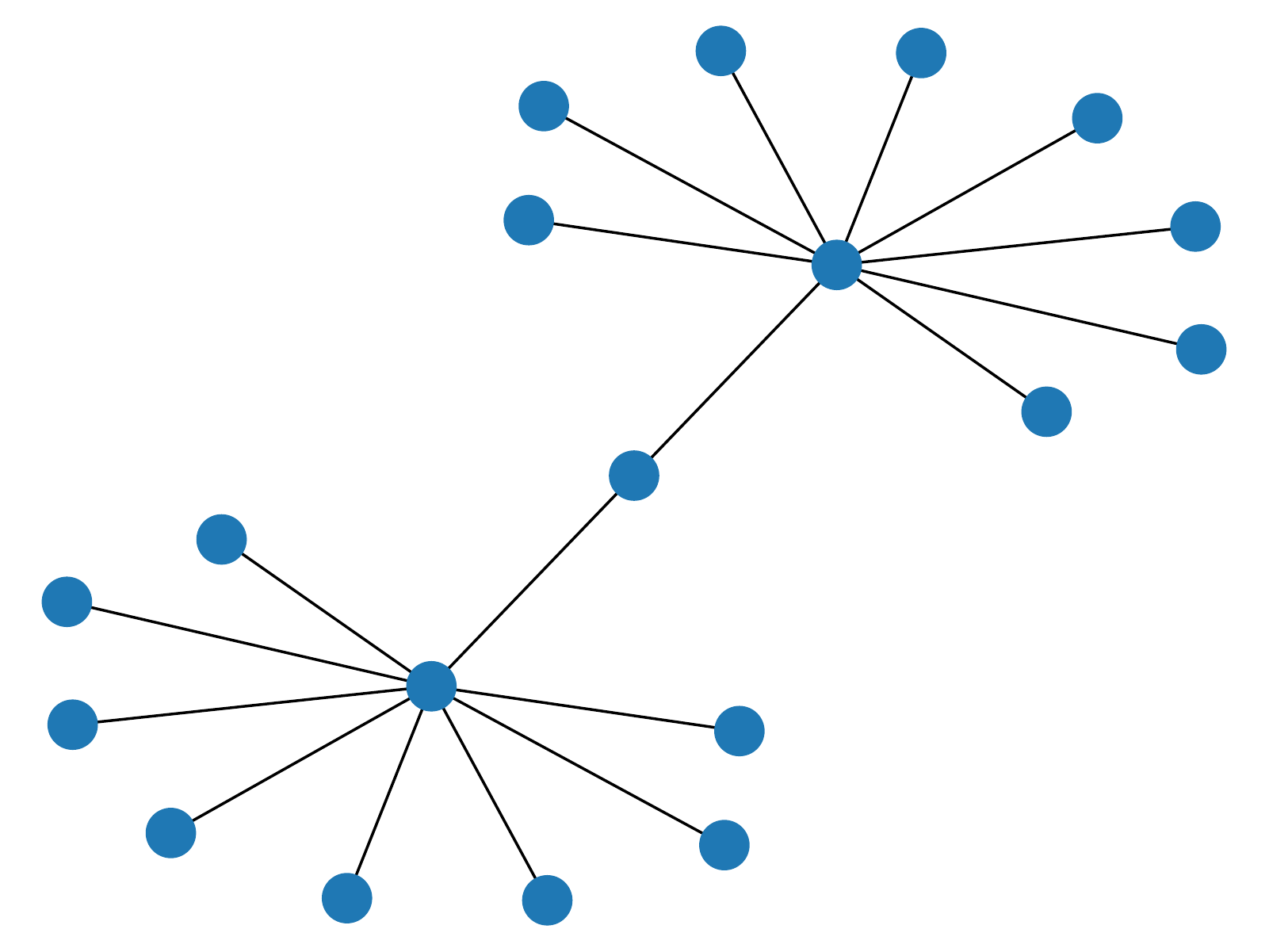}
    \caption{A graph on 19 vertices satisfying $\lambda_1 + \mu < \sqrt{n-1}+1$.}
    \label{fig:aouch}
\end{figure}

It is easy to see that trees are the best counterexample candidates for Conjecture~\ref{conj:aouch}. Indeed, given a graph $G$ with largest matching $M$, we can repeatedly delete edges from $E(G)\setminus M$ without disconnecting the graph. Doing so does not change $\mu(G)$ but decreases the largest eigenvalue. Interestingly, as seen in Figure~\ref{fig:aouch_timeline}, the network also quickly figures out that trees are best, after which it starts decreasing the diameter and converges to the graph in Figure~\ref{fig:aouch}.

\subsection{A conjecture about the proximity and distance eigenvalues of graphs}\label{subsec:aouch2}

 Let $G$ be an $n$-vertex graph. Given an ordering of $V(G)$ as $\{v_1,v_2,\ldots,v_n\}$, the \emph{distance matrix} $D(G)$ of $G$ is the $n\times n$ matrix which has $(i,j)$ entry equal to $d(v_i,v_j)$, where $d(v,w)$ denotes the length of the shortest path between $v$ and $w$ in $G$. The \emph{distance eigenvalues} of $G$ are the eigenvalues $\partial_1\geq \ldots\geq \partial_n$ of $D(G)$. 
 
 The study of the spectral properties of the distance matrix has a long history, reaching back to at least the 70s. Graham and Pollack~\cite{grahampollack} studied the connections between the distance eigenvalues of graphs and a problem in data communication systems and  proved, among others, that the determinant of the distance matrix of a tree only depends on the order of the tree.
 
 \begin{theorem}[Graham-Pollack~\cite{grahampollack}]
 If $T$ is a tree on $n$ vertices with distance matrix $D(T)$, then
 $$\mathrm{det}(D(T)) = (-1)^{n-1}(n-1)2^{n-2}.$$
 \end{theorem}

We will talk more about the spectral properties of the distance matrix of trees in Section~\ref{subsec:peaksfar}.
Motivated by our success in the previous subsection, we will try to disprove a conjecture that is of a very similar flavor to Conjecture~\ref{conj:aouch}. 
Given a connected graph $G$ with $n$ vertices, define its \emph{proximity} $\pi$ as
$$\pi = \frac{1}{n-1}\cdot \min_{v\in V(G)} \sum_{w\in V(G)}d(v,w).$$
Note that $\pi \geq 1$. Merris~\cite{merris1990distance} showed that for every connected graph with diameter $D$ we have $\partial_{\left\lfloor\frac{D}{2}\right\rfloor}>-1$ and hence
$$\pi + \partial_{\left\lfloor\frac{D}{2}\right\rfloor}>0.$$
We will use our methods to refute the following conjecture by Aouchiche and Hansen~\cite{aouchhansen}, that would have been a strengthening of the above result:
\begin{conjecture}[Auchiche--Hansen~\cite{aouchhansen}\label{conj:aouch2}]
Let $G$ be a connected graph on $n\geq 4$ vertices with diameter $D$, proximity $\pi$ and distance spectrum $\partial_1\geq \ldots \geq \partial_n$. Then
$$\pi + \partial_{\left\lfloor \frac{2D}{3} \right\rfloor} > 0.$$
\end{conjecture}

Our strategy for refuting Conjecture~\ref{conj:aouch2} will be the exact same as with Conjecture~\ref{conj:aouch}. The only change will be the reward function: we change $\lambda_1+\mu$ to $\pi + \partial_{\left\lfloor \frac{2D}{3} \right\rfloor}$. Running the algorithm with $n=30$, after a few days the neural network finds the graph depicted in Figure~\ref{fig:aouch2}.

\begin{figure}[htb]
    \centering
    \includegraphics[scale = 0.5]{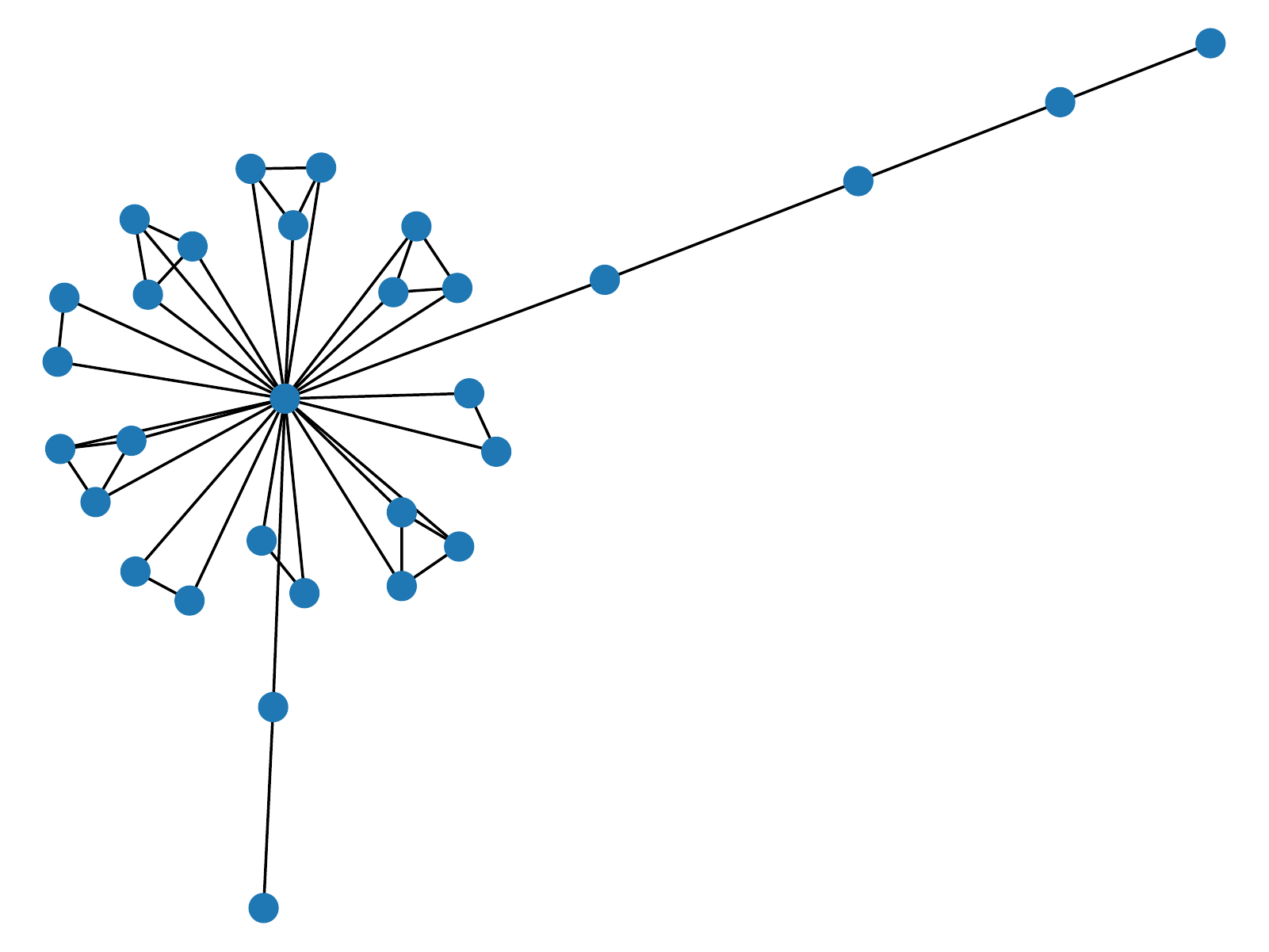}
    \caption{The graph on 30 vertices with smallest value of $\pi + \partial_{\left\lfloor \frac{2D}{3} \right\rfloor}$ found by the network. It has $\pi + \partial_{\left\lfloor \frac{2D}{3} \right\rfloor}\approx 0.4$ so it is not quite a counterexample to Conjecture~\ref{conj:aouch2}, but it tells us very clearly what counterexamples could look like.}
    \label{fig:aouch2}
\end{figure}

The graph in Figure~\ref{fig:aouch2} is likely not optimal for $n=30$: the best graph was still changing when we aborted the algorithm. The reason we chose to terminate the learning process is because we noticed that at this stage, every single graph in the top 10\% of an iteration had essentially the same structure: a long path with a star near its middle, whose neighborhood is partitioned into disjoint cliques. The only thing that varied was the size of these cliques, and having cliques of sizes 2 and 3  for example was ever so slightly better than having some cliques of size 1.

Hence, even though the graph in Figure~\ref{fig:aouch2} is not quite a counterexample to Conjecture~\ref{conj:aouch2}, it gives us a very clear indication on what the structure of graphs minimizing $\pi + \partial_{\left\lfloor \frac{2D}{3} \right\rfloor}$ potentially looks like. Given this information, we can simply increase the number of vertices and vary the clique sizes in this construction until we eventually find a counterexample,  seen in Figure~\ref{fig:aouch2counter}.

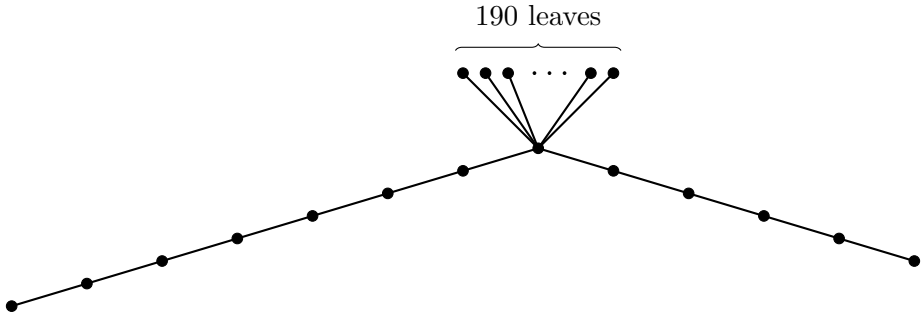
\begin{figure}[htb]
    \centering
    \begin{tikzpicture}
    \draw[ thick] (0,0) -- (1,-0.3);
    
    \draw[thick] (1,-0.3) -- (2,-0.6);
    \draw[fill] (1,-0.3) circle (2pt);
    \draw[thick] (2,-0.6) -- (3,-0.9);
    \draw[fill] (2,-0.6) circle (2pt);
    \draw[thick] (3,-0.9) -- (4,-1.2);
    \draw[fill] (3,-0.9) circle (2pt);
    \draw[thick] (4,-1.2) -- (5,-1.5);
    \draw[fill] (4,-1.2) circle (2pt);
    \draw[fill] (5,-1.5) circle (2pt);
    
    \draw[thick] (0,0) -- (-1,-0.3);
    \draw[fill] (0,0) circle (2pt);
    \draw[thick] (-1,-0.3) -- (-2,-0.6);
    \draw[fill] (-1,-0.3) circle (2pt);
    \draw[thick] (-2,-0.6) -- (-3,-0.9);
    \draw[fill] (-2,-0.6) circle (2pt);
    \draw[thick] (-3,-0.9) -- (-4,-1.2);
    \draw[fill] (-3,-0.9) circle (2pt);
    \draw[thick] (-4,-1.2) -- (-5,-1.5);
    \draw[fill] (-4,-1.2) circle (2pt);
    \draw[thick] (-5,-1.5) -- (-6,-1.8);
    \draw[fill] (-5,-1.5) circle (2pt);
    \draw[thick] (-6,-1.8) -- (-7,-2.1);
    \draw[fill] (-6,-1.8) circle (2pt);
    \draw[fill] (-7,-2.1) circle (2pt);
    
    \draw[thick] (0,0) -- (-1,1);
    \draw[fill] (-1,1) circle (2pt);
    \draw[thick] (0,0) -- (-0.7,1);
    \draw[fill] (-0.7,1) circle (2pt);
    \draw[thick] (0,0) -- (-0.4,1);
    \draw[fill] (-0.4,1) circle (2pt);
    
    \draw[decoration={brace},decorate] (-1.1,1.3) -- node[above=6pt] {190 leaves} (1.1,1.3);
    
    \draw[fill] (-0.05,1) circle (0.5pt);
    \draw[fill] (0.15,1) circle (0.5pt);
    \draw[fill] (0.35,1) circle (0.5pt);
    
    \draw[thick] (0,0) -- (0.7,1);
    \draw[fill] (0.7,1) circle (2pt);
    \draw[thick] (0,0) -- (1,1);
    \draw[fill] (1,1) circle (2pt);

    \end{tikzpicture}
    \caption{A counterexample to Conjecture~\ref{conj:aouch2}.}
    \label{fig:aouch2counter}
\end{figure}

This counterexample is constructed by taking a path on 13 vertices, and attaching $n$ pendant vertices to a vertex adjacent to its midpoint. This graph has diameter 12, and it can be verified by a computer that as soon as $n\geq 190$, this graph satisfies $\pi + \partial_{8} < 0$. We note that this type of graph has already appeared in~\cite{aouchhansen} and~\cite{aouchiche2011proximity}, where they were called \emph{double-tailed comets}, and were shown to minimize the proximity over the class of trees with a given order and diameter, and used as examples to show that $\pi + \partial_{\left\lfloor \frac{2D}{3} \right\rfloor + 1}$ need not always be positive.
From our limited computer experiments it seems plausible that our graph on 203 vertices is in fact (close to being) the smallest counterexample to Conjecture~\ref{conj:aouch2}.

\subsection{The peaks of distance and adjacency polynomials of trees can be far apart}\label{subsec:peaksfar}

Given a tree $T$ on vertex set $\{1,2,\ldots,n\}$, denote its adjacency matrix by $A(T)$. Its distance matrix $D(T)$ has $(i,j)$ entry equal to the graph-theoretic distance of vertices $i$ and $j$. Consider $CPD(T)$, the characteristic polynomial of the distance matrix:

$$CPD(T) = \det (D(T)-x I) = \sum_{k=0}^n \delta_k x^k.$$

In~\cite{aalipour2015proof}, the coefficients $d_k = \frac{2^k}{2^{n-2}}|\delta_k|$ are called the \emph{normalized coefficients} of the characteristic polynomial of the distance matrix of $T$. Graham and Lov\'asz~\cite{grahamlovasz} conjectured that for any tree $T$, the sequence  of normalized coefficients $d_0(T), \ldots , d_{n-2}(T)$ is unimodal and the peak (i.e.~maximum) occurs precisely at index $\lfloor \frac{n}{2}\rfloor$. 

The conjecture regarding the location of the peak was disproved by Collins~\cite{collins} who showed that for paths the peak is at index approximately $\left(1-\frac{1}{\sqrt{5}}\right)n$. In~\cite{aalipour2015proof} the authors confirmed the other part of the Graham-Lov\'asz conjecture by proving that the sequence of normalized coefficients  $d_0(T), \ldots , d_{n-2}(T)$ is indeed always unimodal -- in fact, they proved the stronger statement that this sequence is \emph{log-concave}.

Collins~\cite{collins} considered the characteristic polynomial $CPA(T)$ of the adjacency matrix of trees.  By ignoring the terms with zero coefficients, $CPA(T)$ can be written in the form $CPA(T) = \sum_{i=0}^m a_i x^{\alpha_i}$ where $\alpha_0<\alpha_1<\ldots<\alpha_m$ and $a_i\neq 0$ for all $i$. Collins observed that for the path, $CPA(T)$ has $m=\lfloor \frac{n}{2} \rfloor + 1$ non-zero coefficients whose absolute values form a unimodal sequence with peak at around index $i=\frac{1}{2\sqrt{5}}n $. So the ratio of the position of the peak index to the number of terms in the sequence is also  $1/\sqrt{5}$, just like for the distance matrix!  

Given a tree $T$, denote by $p_A(T)$  the position of the peak index of the absolute values of the non-zero coefficients of $CPA(T)$, and denote by $p_D(T)$  the index of the peak of the  normalized coefficients of $CPD(T)$. Let $m(T)$ denote the number of non-zero coefficients of $CPA(T)$ and $n(T)$ the number of terms in $CPD(T)$ (which is just the number of vertices of $T$ plus one). What the above results say precisely, is that for the path $P_n$ we have $$\lim_{n\rightarrow\infty}\frac{p_D(P_n)}{n(P_n)} = \lim_{n\rightarrow\infty}\left(1 - \frac{p_T(P_n)}{m(P_n)} \right)= 1-\frac{1}{\sqrt{5}}.$$
That is, the two peaks in the case of the path are in the ``same'' place (when counting the location of the peak from opposite ends of the two sequences of coefficients). Motivated by this, and that this equality holds for the star as well, Collins~\cite{collins} made the following conjecture:
\begin{conjecture}[Collins~\cite{collins}]
The sequence of absolute values of the non-zero coefficients of $CPA(T)$ form a unimodal sequence, and its peak is at the same place as the peak of the normalized coefficients of $CPD(T)$.
\end{conjecture}

  We shall only focus on the conjecture regarding the location of the peaks, and disprove it in a strong sense. We represent trees using their Pr\"ufer codes as words of length $n-2$ from an alphabet of size $n$. By optimizing the reward function $f(T) =\bigg | \frac{p_A(T)}{m(T)} - \left(1-\frac{p_D(T)}{n(T)}\right) \bigg |$ using the cross-entropy method, we  find that the tree on $48$ vertices in Figure~\ref{fig:collins} has $\frac{p_D(T)}{n(T)}=\frac{1}{2}$ and $\frac{p_A(T)}{m(T)}=\frac{1}{7}$.  

\begin{figure}[hbt]
    \centering
    \includegraphics[scale=0.5]{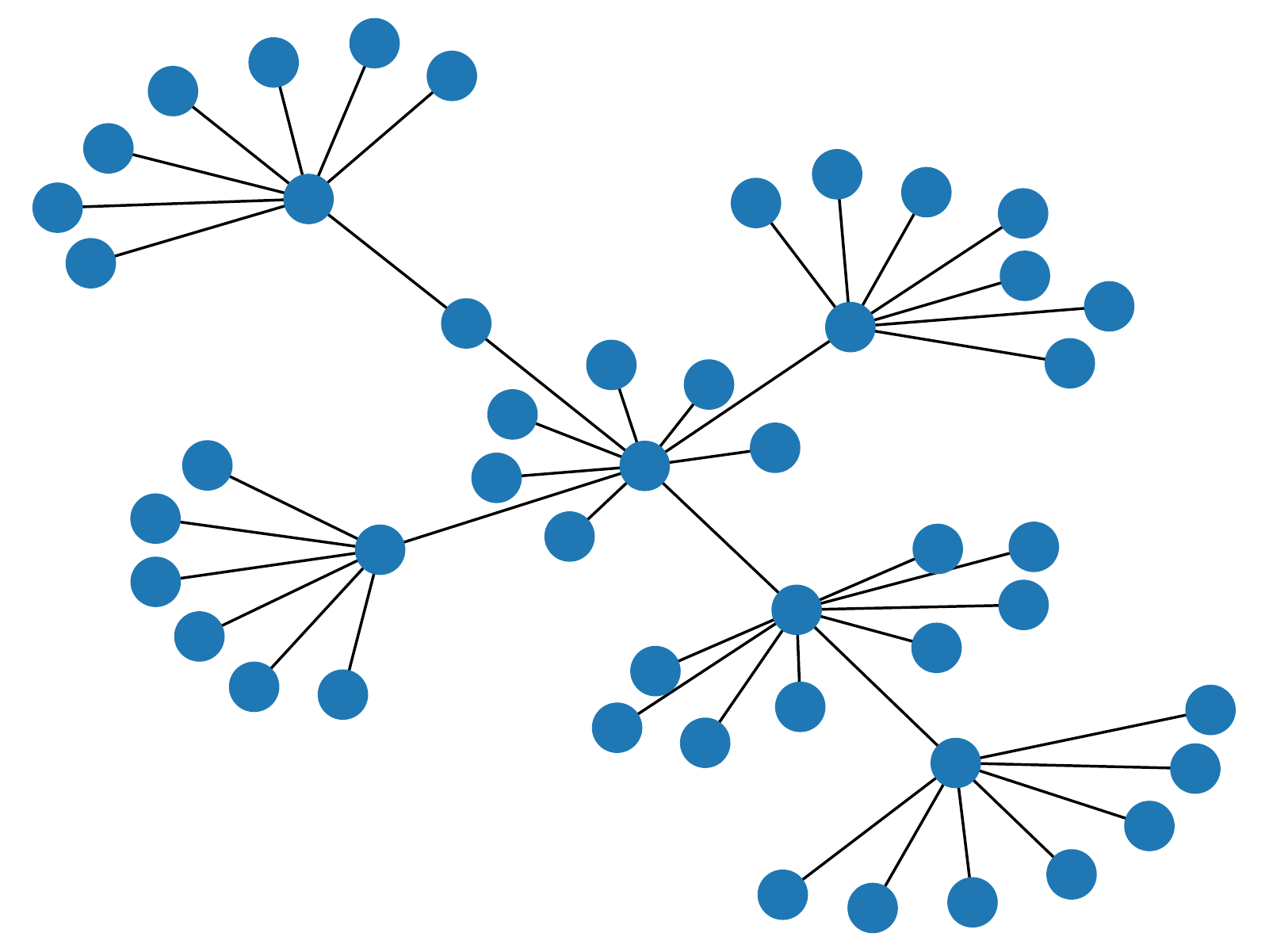}
    \caption{The tree $T$ on $48$ vertices found by the network, where the two peaks of the polynomials are furthest apart. It has $\frac{p_A(T)}{m(T)}=\frac{1}{7}$ and $\frac{p_D(T)}{n(T)}=\frac{1}{2}$. Observe that it has six high degree vertices, and every edge in the tree is adjacent to at least one of these six vertices.}
    \label{fig:collins}
\end{figure}

Once we know what the extremal examples approximately look like, the following theorem is much easier to prove. It refutes Collins's conjecture in a strong sense.
\begin{theorem}\label{thm:collins}
There exists an infinite sequence of trees $(T_k)_{k\in\mathbb{N}}$ satisfying $$\lim_{k\rightarrow\infty}n(T_k) = \lim_{k\rightarrow\infty}m(T_k) = \infty,$$ such that $f(T_k)\geq 0.3$ for all $k\in\mathbb{N}$.
\end{theorem}
Theorem~\ref{thm:collins} shows that even if we assume that the two sequences have lots of non-zero entries, the two peaks can be very far apart. This avoids issues about when $m(T)$ is small, and it is somewhat ambiguous what one means with the ratio of the peak index. Note that for example when $T$ is a star, then $CPA(T)$ only has two non-zero entries~\cite{collins}.

Let $n,d\geq 3$ be integers such that $d|n$. Let $T_{n,d}$ be obtained by first taking a path on $d$ vertices, and attaching $\frac{n-d}{d}$ leaves to each vertex. We show that if $d$ is an arbitrary  constant and $n$ is sufficiently large compared to $d$, then the non-zero coefficients of $CPA(T_{n,d})$ form a monotone sequence, so their peak is as far from the halfway point as possible.

\begin{proposition}\label{prop:peakbounds}
If $n\geq d^2$ then $\frac{p_A(T_{n,d})}{m(T_{n,d})}=\frac{1}{d+1}$.
\end{proposition}
\begin{proof}
    Recall the fact (see e.g~\cite{cvetkovic1980spectra}) that we can describe the characteristic polynomial of the adjacency matrix of a tree in terms of the number of different-sized matchings in the tree:
    \begin{equation*}
        a_{n-k}=\begin{cases} (-1)^{n+k/2} \cdot N_{\frac{1}{2}k}(T)&\mbox{if } k \text{ even,} \\
0 & \mbox{otherwise.} \end{cases} 
    \end{equation*}
    Here $N_{i}(T)$ denotes the number of matchings of size $i$ in $T$. Observe that the largest matching in $T_{n,d}$ has size $d$, and so $CPA(T_{n,d})$ has $d+1$ non-zero coefficients, one of them corresponding to the empty matching. For every $0\leq k \leq d$ we have the following easy bounds on the coefficients:
    $$\binom{d}{k}\left(\frac{n-d}{d}\right)^k\leq |a_{n-2k}|\leq  \binom{d}{k} \left(\frac{n-d}{d} + 2\right)^k.$$
    The lower bound comes from only counting those $k$-matchings where every edge of the matching meets the path we used when constructing $T_{n,d}$ in precisely one of its $d$ vertices. For the upper bound, observe that every edge in the tree is adjacent to the path on $d$ vertices used in the construction, and every vertex on this path has at most $\frac{n-d}{d} + 2$ edges.
    
    It follows that if $n\geq d^2-d+1$ then the sequence of absolute values of coefficients is strictly decreasing. Thus the peak is at very first index, i.e.~$a_{n-2d}$, and we have $\frac{p_A(T_{n,d})}{m(T_{n,d})}=\frac{1}{d+1}$.
\end{proof}

It remains to bound the peak of the distance coefficients. While the coefficients of the adjacency polynomial only depend on the number of matchings of certain sizes in the tree, calculating the coefficients of the distance polynomial is significantly more involved (see~\cite{grahamlovasz}). It seems that when $n$ is large enough compared to $d$, we always have $\frac{p_D(T_{n,d})}{n(T_{n,d})} = \frac{\lfloor n/2 \rfloor}{n}\approx \frac{1}{2}$. We do not know how to prove this. Instead we will use a result from~\cite{aalipour2015proof}, which states that $p_D(T)\leq \lceil \frac{2}{3}n(T)\rceil$ for every tree $T$.

\begin{proof}[Proof of Theorem~\ref{thm:collins}]
Consider the sequence of trees $T_{d^2,d}$ with $d=31,32, \ldots$. By Proposition~\ref{prop:peakbounds}, for each of these trees we have $\frac{p_A(T_{d^2,d})}{m(T_{d^2,d})}=\frac{1}{d+1}$. By the sentence preceding this proof, we have $\frac{p_D(T_{d^2,d})}{n(T_{d^2,d})}\leq \frac{2}{3} + \frac{1}{n(T_{d^2,d})}$. As $n(T_{d^2,d})=d^2+1$ and as seen in the proof of Proposition~\ref{prop:peakbounds} we also have $m(T_{d^2,d}) = d+1$, both $n(T_{d^2,d})$ and $m(T_{d^2,d})$ go to infinity as $d$ increases, as required. The distance between the peaks satisfies
$$f\left(T_{d^2,d}\right) \geq 1 - \frac{1}{d+1} - \frac{2}{3} - \frac{1}{d^2+1} \geq 0.3$$
for all $d\geq 31.$
\end{proof}

Note that it is not necessarily the case that the trees $T_{n,d}$ maximize the distance of the peaks, for a given number of vertices. One counterexample to this is at $n=30$. The tree in Figure~\ref{fig:tree30} has $f(T)=\frac{1}{3}$, which is strictly larger than $f(T_{30,d})$ for any $d$. The reason is that the tree on the left has more 5-matchings (4705) than 4-matchings (4601), whereas this is not the case for e.g.~$T_{30,5}$ (where the corresponding numbers are 3125 and 3625).

\begin{figure}[hbt]
    \centering
    \includegraphics[scale=0.4]{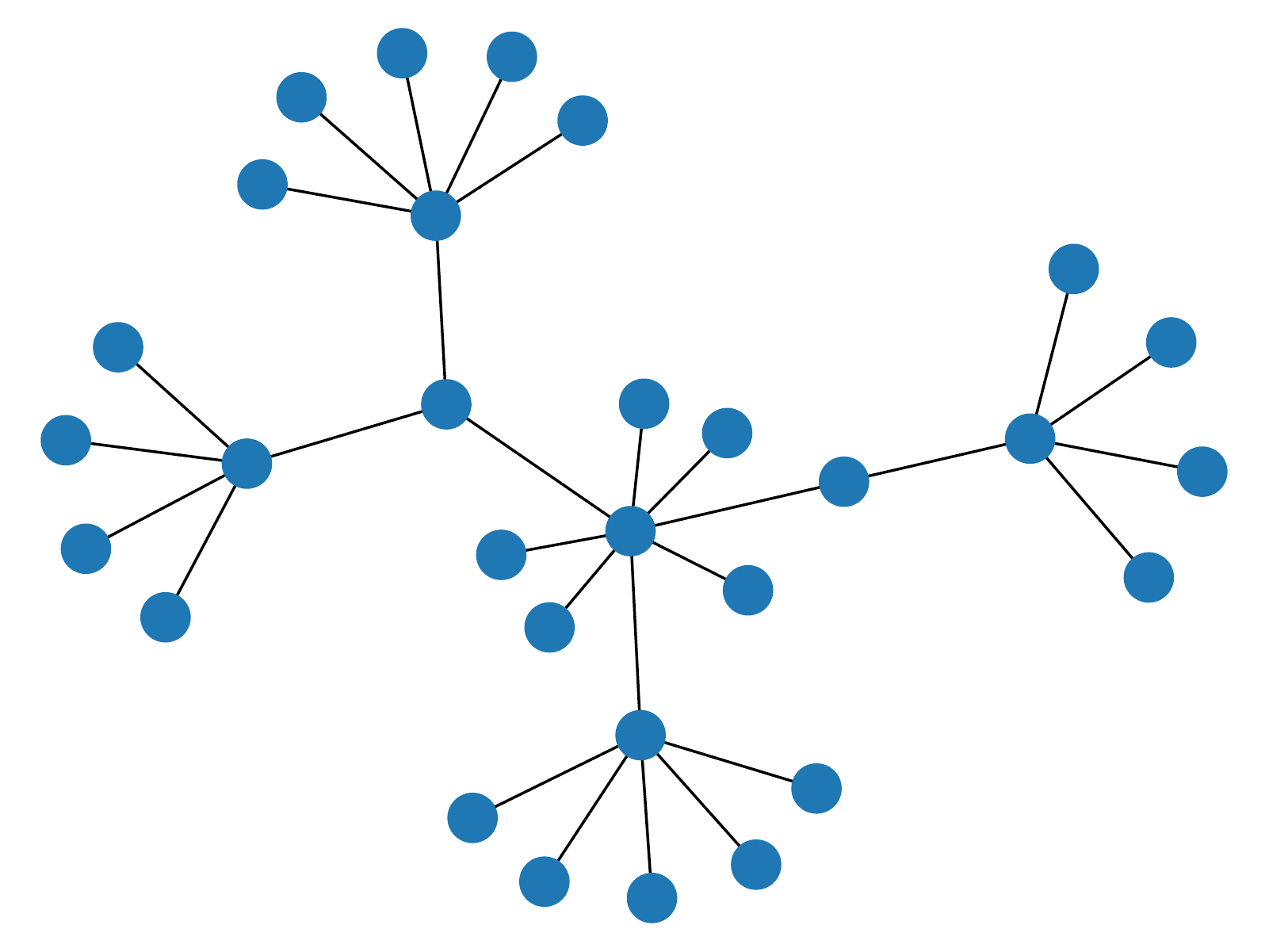}
    \hspace{0.5in}
    \includegraphics[scale=0.4]{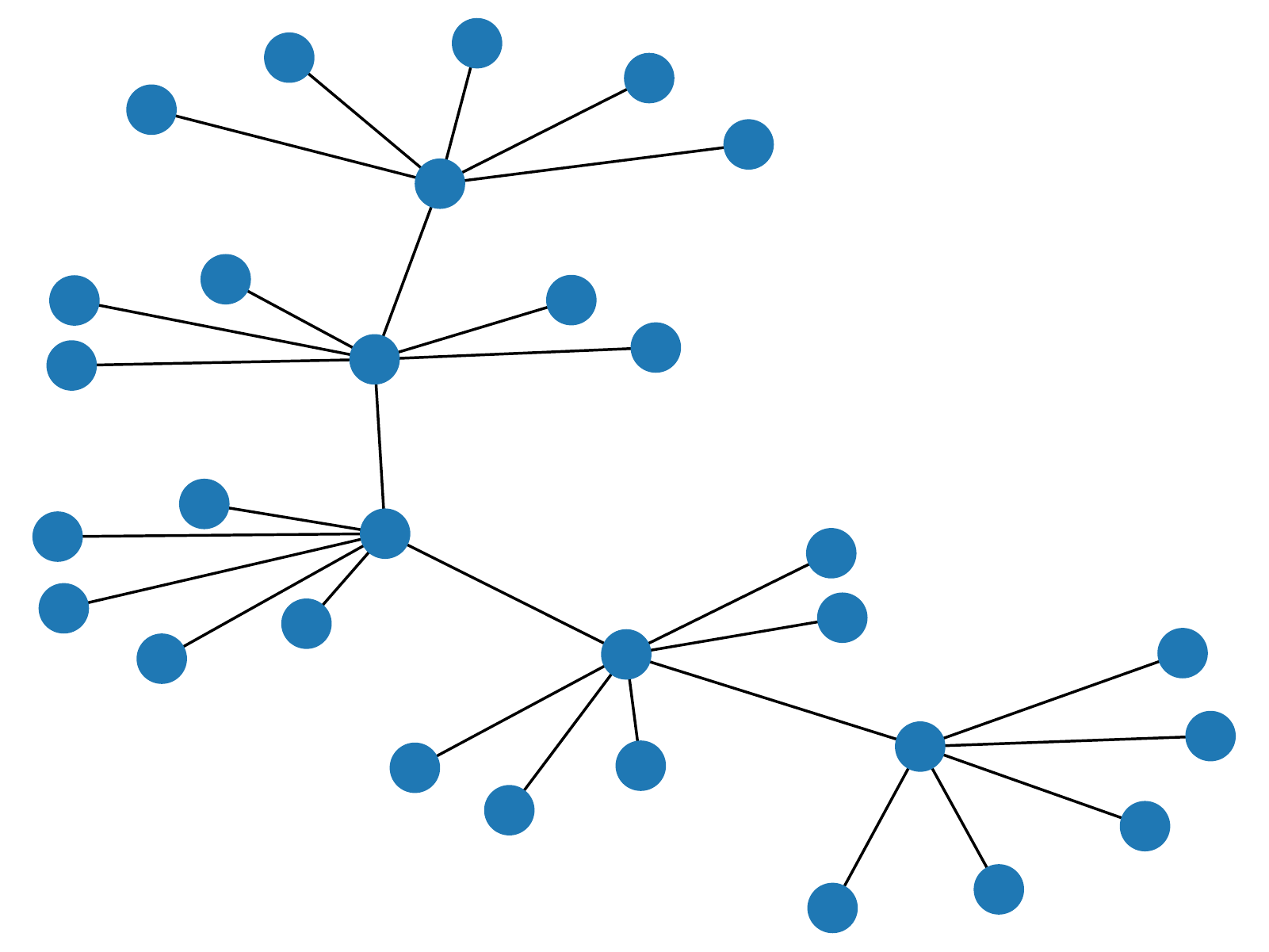}
    \caption{The tree $T$ on the left has $f(T)=\frac{1}{3}$. For $T_{30,5}$ on the right, we have $f(T_{30,5})=\frac{1}{6}$.}
    \label{fig:tree30}
\end{figure}

There are many things that are still not known about the location of the peaks of various tree polynomials. We mention only one open conjecture in the area, attributed to Peter Shor. Note that the lower bound corresponds to the star and the upper bound to the path.

\begin{conjecture}[Shor~\cite{collins}]
For any tree on $n$ vertices, the peak $p_D(T)$ satisfies
$$\Big\lfloor \frac{n}{2}\Big\rfloor \leq p_D(T) \leq \Bigg\lceil n\left(1-\frac{1}{\sqrt{5}}\right)\Bigg\rceil.$$
\end{conjecture}

\subsection{Transmission regularity and the distance Laplacian}\label{subsec:transmission}

We continue our study of the spectrum of various graph matrices associated to graphs. We will consider the following question: what information about the graph can be recovered from the spectrum of such a matrix? Here we will focus on the \emph{distance Laplacian} of $G$, denoted by $\mathcal{D}^L(G)$. 

Given a connected graph $G$ with vertex set $V(G)=\{v_1,\ldots,v_n\}$, we denote by $t(v_i)=\sum_{u\in V(G)}d(v_i,u)$ the \emph{transmission of vertex $v_i$}, i.e.~the sum of distances from $v_i$ to all other vertices in $G$. We say that $G$ is \emph{transmission regular} if $t(v_i)=t(v_j)$ for all $i,j$. Let $T(G)$ be the diagonal matrix with entries $\{ t(v_i): i\in [n] \}$, and let $D(G)$ be the distance matrix of $G$ with $i,j$-entry equal to $d(v_i,v_j)$. The distance Laplacian is then defined as $\mathcal{D}^L(G)= T(G)-D(G)$. Denote the spectrum of $\mathcal{D}^L(G)$, i.e.~the multiset of its eigenvalues, by $\text{spec}_{\mathcal{D}^L}(G)$.

Given a graph property $\mathcal{P}$ , we say $\mathcal{P}$ is preserved by $\mathcal{D}^L$-cospectrality if $\text{spec}_{\mathcal{D}^L}(G)=\text{spec}_{\mathcal{D}^L}(H)$ implies $\mathcal{P}(G)=\mathcal{P}(H)$. A natural question to ask is which graph properties are preserved by $\mathcal{D}^L$-cospectrality?

This question has a long history, for an overview on the extensive research on this topic we refer the reader to the recent survey~\cite{surveydistance} and the references therein, in particular~\cite{hogben}. Hogben--Reinhart~\cite{surveydistance} put a lot of emphasis on the spectral properties of transmission regular graphs -- indeed, as shown in Table~7.2 in their survey, it is the last of the natural graph properties considered in~\cite{surveydistance} for which is not known whether it is preserved by $\mathcal{D}^L$-cospectrality. 

Our main result in this section is to fill in this gap, by showing that transmission regularity is not preserved by $\mathcal{D}^L$-cospectrality. That is, there exist graphs $G,H$ such that $G$ is transmission regular, $H$ is not transmission regular, and  $\text{spec}_{\mathcal{D}^L}(G)=\text{spec}_{\mathcal{D}^L}(H)$. Such two graphs can be seen in Figure~\ref{fig:cospectral}. The characteristic polynomials of their distance Laplacians are the same, they are given by $x^{12} - 216 x^{11} + 21188 x^{10} - 1245904 x^9 + 48797440 x^8 - 1336652544 x^7 + 26129121472 x^6 - 364516883456 x^5 + 3556516628224 x^4 - 23113129559040 x^3 + 90045806284800 x^2 - 159318669312000 x$.

\begin{figure}[hbt]
    \centering
    \includegraphics[scale=0.4]{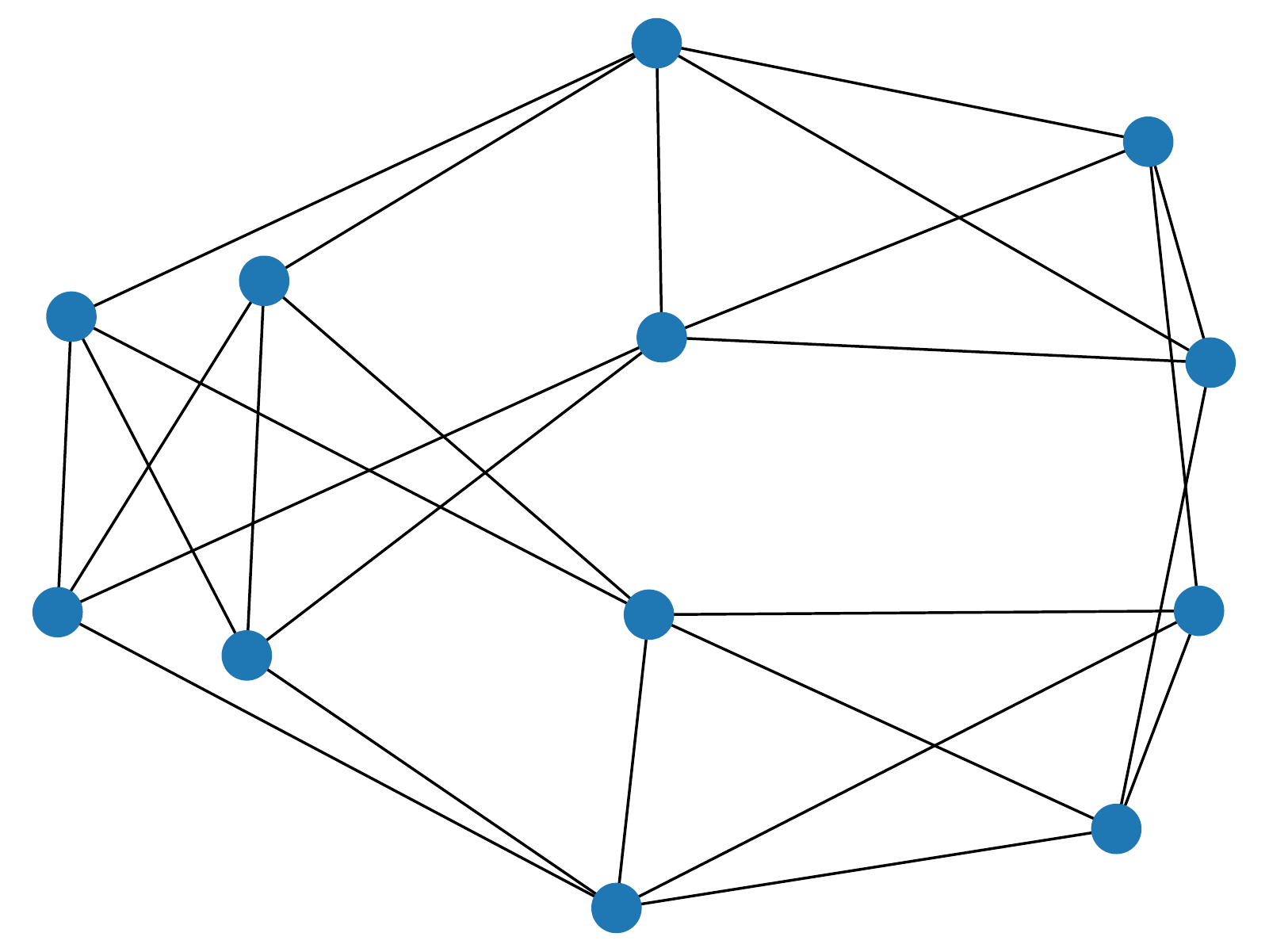}
    \includegraphics[scale=0.4]{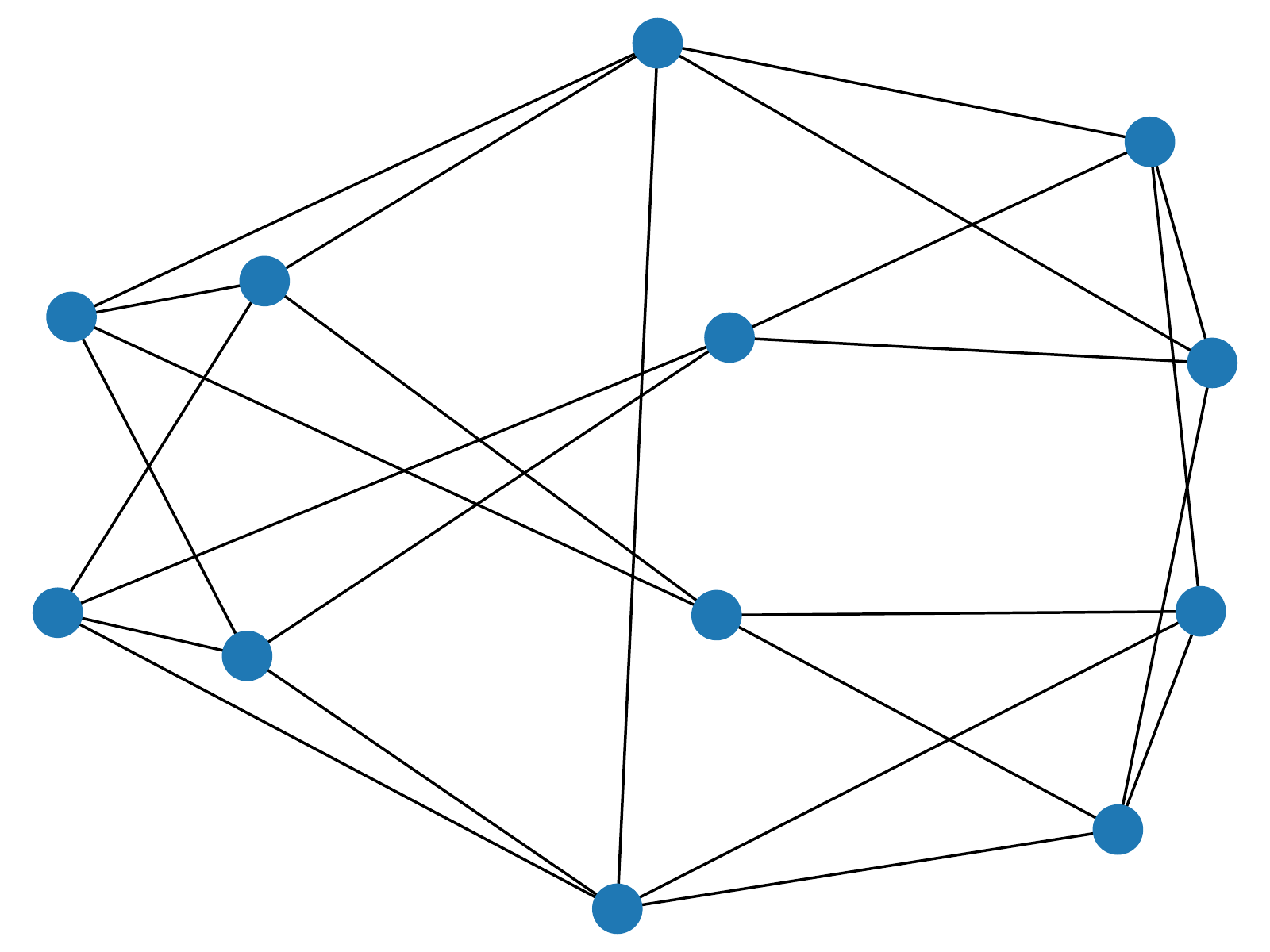}
    \caption{The graph on the left is transmission regular, whereas the graph on the right is not. The characteristic polynomials of their distance Laplacians are the same, so they are $\mathcal{D}^L$-cospectral.}
    \label{fig:cospectral}
\end{figure}

The construction is not unique. There are many different ways one can design a reward function to use with the cross-entropy method in order to try to produce a cospectral pair of graphs as in Figure~\ref{fig:cospectral}. None of the reward functions we tried performed particularly well, in the end it seemed more of a coincidence that some of the runs of the algorithm stumbled upon a construction. It is likely that other algorithms are much better suited for this problem.

\subsection{The permanent of 312-pattern avoiding 0-1 matrices}\label{subsec:perm312}

Pattern avoidance, and permutation avoidance in particular, is a central topic in combinatorics. In recent decades the study of permutation patterns has become a
discipline in its own right, with hundreds of published papers. For an overview, we refer the reader to the books~\cite{permbook1,permbook2}, surveys~\cite{permsurvey1,permsurvey2}, and the many applications in Tenner’s database~\cite{tenner}.

Given a positive integer $n$, denote by $S_n$ the set of permutations of $[n]=\{1,2,\ldots,n\}$. Let $\sigma\in S_n$ and $\tau\in S_k$ for some integers $k\leq n$. We say the permutation $\sigma$ \emph{contains the pattern} $\tau$ if there exist integers $1\leq x_1\leq x_2\leq \ldots \leq x_k\leq n$ such that for $1\leq i,j\leq k$ we have
$$\sigma(x_i) < \sigma(x_j) \text{ if and only if } \tau(i)< \tau(j);$$ otherwise, we say $\sigma$ \emph{avoids} $\tau$. Denote by $S_n(\tau)$ the set of $\sigma\in S_n$ that avoid $\tau$.

The above definition can be extended to 0-1 matrices as well. Let $A$ and $P$ be 0-1 matrices. We say that $A$ \emph{contains} the $k \times \ell$ matrix $P=(p_{ij})$ if there exists a $k \times \ell$ submatrix $D = (d_{ij})$ of $A$ with $d_{ij}\geq p_{ij}$ for all $i,j$; otherwise we say that $A$ \emph{avoids} $P$.

To any permutation $\sigma\in S_n$ we can associate a permutation matrix, which is the $n\times n$ 0-1 matrix with ones precisely in positions $(i,\sigma(i))$ for $i=1,2,\ldots,n$. We say a matrix $A$ \emph{avoids} $\sigma$ if $A$ avoids the permutation matrix associated with $\sigma$.

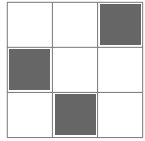
\begin{figure}[hbt]
    \centering
    \begin{tikzpicture}[scale=2]
\draw[step=0.3cm,gray,thin] (0,0) grid (0.8999999999999999,0.8999999999999999);
\fill[black!60!white] (0.315,0.285) rectangle (0.585,0.0149999999999999);
\fill[black!60!white] (0.015,0.585) rectangle (0.285,0.3149999999999999);
\fill[black!60!white] (0.615,0.885) rectangle (0.885,0.6149999999999999);
\end{tikzpicture}

    \caption{The pattern 312}
    \label{fig:312}
\end{figure}

 For a 0-1 matrix $P$ let $f(n, P)$ be the maximum number of $1$-entries in an $n \times n$ 0-1 matrix avoiding~$P$. F\"uredi and Hajnal conjectured~\cite{furedihajnal} that for all permutation matrices $P$ we have $f(n, P)= O(n)$. A cornerstone result of Marcus and Tardos~\cite{marcustardos} is that the F\"uredi--Hajnal conjecture is true. By a result of Klazar~\cite{klazar}, this also implied the celebrated Stanley-Wilf conjecture, which states that for all permutations $\tau$ there exists a constant $c_\tau$ such that $|S_n(\tau)|\leq c_\tau^n$. A breakthrough result by Fox~\cite{fox} shows that $c_\tau=2^{k^{\theta(1)}}$ for almost all $\tau\in S_k$. 
 
 The \emph{permanent} of an $n\times n$ matrix $A$ is defined as
 $$\text{per}(A)=\sum_{\sigma \in S_n}\prod_{i=1}^n a_{i,\sigma(i)}.$$
 Observe that if $J_n$ is the $n\times n$ matrix with all 1 entries and we restrict the sum to $S_n(\tau)$ then the resulting permanent is equal to $|S_n(\tau)|$. Perhaps motivated by this observation, Brualdi and Cao asked the following beautiful question.  Let $M_n(\sigma)$ denote the set of $n\times n$ 0-1 matrices that avoid $\sigma$.
 \begin{question}[\label{ques:brualdi}Brualdi--Cao~\cite{brualdi}]
Given an integer $n$ and a permutation $\sigma$, what is the value of $$f_\sigma(n)=\max\{\text{per}(A) : A \in M_n(\sigma)\}?$$
 \end{question}
 
Brualdi and Cao observed that this question is already interesting when $\sigma\in S_3$ is a permutation on 3 elements. They guessed that the answer to Question~\ref{ques:brualdi} when $\sigma=312$ is given by the following matrix, illustrated here for $n=5$:
$$\begin{bmatrix}
1 & 1 & 1 & 0 & 0 \\
1 & 1 & 1 & 1 & 0 \\
1 & 0 & 1 & 1 & 1 \\
1 & 0 & 0 & 1 & 1 \\
1 & 0 & 0 & 0 & 1 \\
\end{bmatrix}$$
For $n=5$ this matrix has permanent 12, and in general the permanent is given by the $(n+2)$-th Fibonacci number minus one.

By using the cross-entropy method together with some ad hoc methods we managed to find better constructions giving improved lower bounds for $f_{312}(n)$, disproving the above guess of Brualdi--Cao.  These constructions can be seen in Figure~\ref{fig:brualdi1}. We believe these are optimal up to at least $n= 10$, but we only have a computer-assisted proof showing their optimality for $n\leq 8$. This proves that the sequence $f_{312}(n)$ has the unexpected initial segment
$$1,\quad 2,\quad 4,\quad 8,\quad 16,\quad 32,\quad 64,\quad 120.$$

These constructions are, for most values of $n$, not unique. The neural network-based approach performed well by itself for about $n\leq 10$, but for larger $n$ the numerous calculations of matrix permanents inherent to the cross-entropy method proved too expensive. For $n=11,12$ we used ad hoc methods to improve the constructions given by the algorithm. The construction for $n=13$ was found under the very strong assumption that it should look similar to the previous matrices, and it is likely not optimal. It is generally believed that the permanent cannot be computed in polynomial time, as a seminal result of Valiant~\cite{valiant} shows that computing the permanent of 0-1 matrices is \#P-complete.  We do not know if there is a fast algorithm for computing the permanent if the 0-1 matrix is assumed to be 312-avoiding.

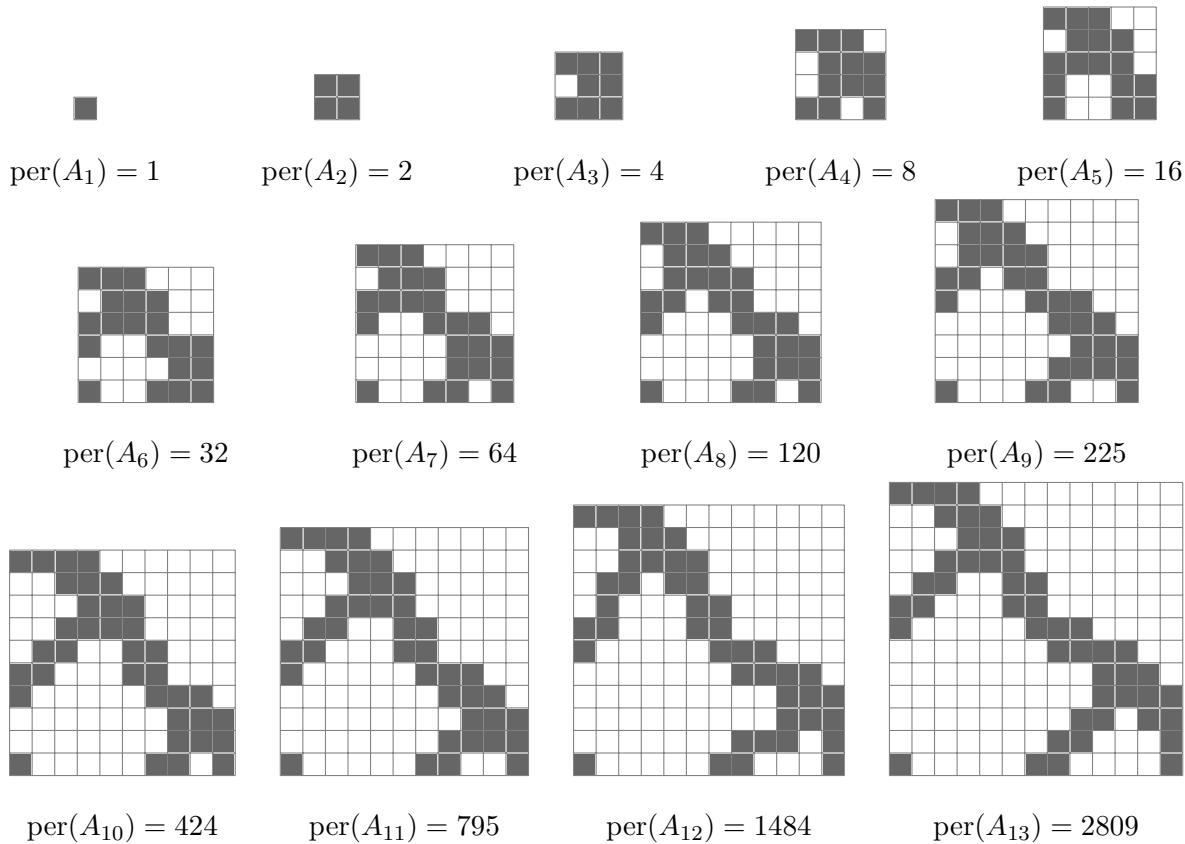
\begin{figure}[hbt]
    \centering
    
\begin{tikzpicture}
\draw[step=0.3cm,gray,thin] (0,0) grid (0.3,0.3);
\fill[black!60!white] (0.015,0.285) rectangle (0.285,0.015000000000000013);
\node[] at (0.15,-0.7) {  $\text{per}\left(A_{1}\right)=1$};\end{tikzpicture}
\hspace{0.7 cm}
\begin{tikzpicture}
\draw[step=0.3cm,gray,thin] (0,0) grid (0.6,0.6);
\fill[black!60!white] (0.015,0.585) rectangle (0.285,0.315);
\fill[black!60!white] (0.315,0.585) rectangle (0.585,0.315);
\fill[black!60!white] (0.015,0.285) rectangle (0.285,0.015000000000000013);
\fill[black!60!white] (0.315,0.285) rectangle (0.585,0.015000000000000013);
\node[] at (0.3,-0.7) {  $\text{per}\left(A_{2}\right)=2$};\end{tikzpicture}
\hspace{0.7 cm}
\begin{tikzpicture}
\draw[step=0.3cm,gray,thin] (0,0) grid (0.8999999999999999,0.8999999999999999);
\fill[black!60!white] (0.015,0.885) rectangle (0.285,0.6149999999999999);
\fill[black!60!white] (0.315,0.885) rectangle (0.585,0.6149999999999999);
\fill[black!60!white] (0.6149999999999999,0.885) rectangle (0.885,0.6149999999999999);
\fill[black!60!white] (0.315,0.585) rectangle (0.585,0.315);
\fill[black!60!white] (0.6149999999999999,0.585) rectangle (0.885,0.315);
\fill[black!60!white] (0.015,0.285) rectangle (0.285,0.015000000000000013);
\fill[black!60!white] (0.315,0.285) rectangle (0.585,0.015000000000000013);
\fill[black!60!white] (0.6149999999999999,0.285) rectangle (0.885,0.015000000000000013);
\node[] at (0.44999999999999996,-0.7) {  $\text{per}\left(A_{3}\right)=4$};\end{tikzpicture}
\hspace{0.7 cm}
\begin{tikzpicture}
\draw[step=0.3cm,gray,thin] (0,0) grid (1.2,1.2);
\fill[black!60!white] (0.015,1.185) rectangle (0.285,0.9149999999999999);
\fill[black!60!white] (0.315,1.185) rectangle (0.585,0.9149999999999999);
\fill[black!60!white] (0.6149999999999999,1.185) rectangle (0.885,0.9149999999999999);
\fill[black!60!white] (0.315,0.885) rectangle (0.585,0.6149999999999999);
\fill[black!60!white] (0.6149999999999999,0.885) rectangle (0.885,0.6149999999999999);
\fill[black!60!white] (0.9149999999999999,0.885) rectangle (1.185,0.6149999999999999);
\fill[black!60!white] (0.315,0.585) rectangle (0.585,0.315);
\fill[black!60!white] (0.6149999999999999,0.585) rectangle (0.885,0.315);
\fill[black!60!white] (0.9149999999999999,0.585) rectangle (1.185,0.315);
\fill[black!60!white] (0.015,0.285) rectangle (0.285,0.015000000000000013);
\fill[black!60!white] (0.315,0.285) rectangle (0.585,0.015000000000000013);
\fill[black!60!white] (0.9149999999999999,0.285) rectangle (1.185,0.015000000000000013);
\node[] at (0.6,-0.7) {  $\text{per}\left(A_{4}\right)=8$};\end{tikzpicture}
\hspace{0.7 cm}
\begin{tikzpicture}
\draw[step=0.3cm,gray,thin] (0,0) grid (1.5,1.5);
\fill[black!60!white] (0.015,1.485) rectangle (0.285,1.2149999999999999);
\fill[black!60!white] (0.315,1.485) rectangle (0.585,1.2149999999999999);
\fill[black!60!white] (0.6149999999999999,1.485) rectangle (0.885,1.2149999999999999);
\fill[black!60!white] (0.315,1.185) rectangle (0.585,0.9149999999999999);
\fill[black!60!white] (0.6149999999999999,1.185) rectangle (0.885,0.9149999999999999);
\fill[black!60!white] (0.9149999999999999,1.185) rectangle (1.185,0.9149999999999999);
\fill[black!60!white] (0.015,0.885) rectangle (0.285,0.6149999999999999);
\fill[black!60!white] (0.315,0.885) rectangle (0.585,0.6149999999999999);
\fill[black!60!white] (0.6149999999999999,0.885) rectangle (0.885,0.6149999999999999);
\fill[black!60!white] (0.9149999999999999,0.885) rectangle (1.185,0.6149999999999999);
\fill[black!60!white] (0.015,0.585) rectangle (0.285,0.315);
\fill[black!60!white] (0.9149999999999999,0.585) rectangle (1.185,0.315);
\fill[black!60!white] (1.2149999999999999,0.585) rectangle (1.485,0.315);
\fill[black!60!white] (0.015,0.285) rectangle (0.285,0.015000000000000013);
\fill[black!60!white] (0.9149999999999999,0.285) rectangle (1.185,0.015000000000000013);
\fill[black!60!white] (1.2149999999999999,0.285) rectangle (1.485,0.015000000000000013);
\node[] at (0.75,-0.7) {  $\text{per}\left(A_{5}\right)=16$};\end{tikzpicture}

\begin{tikzpicture}
\draw[step=0.3cm,gray,thin] (0,0) grid (1.7999999999999998,1.7999999999999998);
\fill[black!60!white] (0.015,1.785) rectangle (0.285,1.515);
\fill[black!60!white] (0.315,1.785) rectangle (0.585,1.515);
\fill[black!60!white] (0.6149999999999999,1.785) rectangle (0.885,1.515);
\fill[black!60!white] (0.315,1.485) rectangle (0.585,1.2149999999999999);
\fill[black!60!white] (0.6149999999999999,1.485) rectangle (0.885,1.2149999999999999);
\fill[black!60!white] (0.9149999999999999,1.485) rectangle (1.185,1.2149999999999999);
\fill[black!60!white] (0.015,1.185) rectangle (0.285,0.9149999999999999);
\fill[black!60!white] (0.315,1.185) rectangle (0.585,0.9149999999999999);
\fill[black!60!white] (0.6149999999999999,1.185) rectangle (0.885,0.9149999999999999);
\fill[black!60!white] (0.9149999999999999,1.185) rectangle (1.185,0.9149999999999999);
\fill[black!60!white] (0.015,0.885) rectangle (0.285,0.6149999999999999);
\fill[black!60!white] (0.9149999999999999,0.885) rectangle (1.185,0.6149999999999999);
\fill[black!60!white] (1.2149999999999999,0.885) rectangle (1.485,0.6149999999999999);
\fill[black!60!white] (1.515,0.885) rectangle (1.785,0.6149999999999999);
\fill[black!60!white] (1.2149999999999999,0.585) rectangle (1.485,0.315);
\fill[black!60!white] (1.515,0.585) rectangle (1.785,0.315);
\fill[black!60!white] (0.015,0.285) rectangle (0.285,0.015000000000000013);
\fill[black!60!white] (0.9149999999999999,0.285) rectangle (1.185,0.015000000000000013);
\fill[black!60!white] (1.2149999999999999,0.285) rectangle (1.485,0.015000000000000013);
\fill[black!60!white] (1.515,0.285) rectangle (1.785,0.015000000000000013);
\node[] at (0.8999999999999999,-0.7) {  $\text{per}\left(A_{6}\right)=32$};\end{tikzpicture}
\hspace{1.0 cm}
\begin{tikzpicture}
\draw[step=0.3cm,gray,thin] (0,0) grid (2.1,2.1);
\fill[black!60!white] (0.015,2.085) rectangle (0.285,1.815);
\fill[black!60!white] (0.315,2.085) rectangle (0.585,1.815);
\fill[black!60!white] (0.6149999999999999,2.085) rectangle (0.885,1.815);
\fill[black!60!white] (0.315,1.785) rectangle (0.585,1.515);
\fill[black!60!white] (0.6149999999999999,1.785) rectangle (0.885,1.515);
\fill[black!60!white] (0.9149999999999999,1.785) rectangle (1.185,1.515);
\fill[black!60!white] (0.015,1.485) rectangle (0.285,1.2149999999999999);
\fill[black!60!white] (0.315,1.485) rectangle (0.585,1.2149999999999999);
\fill[black!60!white] (0.6149999999999999,1.485) rectangle (0.885,1.2149999999999999);
\fill[black!60!white] (0.9149999999999999,1.485) rectangle (1.185,1.2149999999999999);
\fill[black!60!white] (0.015,1.185) rectangle (0.285,0.9149999999999999);
\fill[black!60!white] (0.9149999999999999,1.185) rectangle (1.185,0.9149999999999999);
\fill[black!60!white] (1.2149999999999999,1.185) rectangle (1.485,0.9149999999999999);
\fill[black!60!white] (1.515,1.185) rectangle (1.785,0.9149999999999999);
\fill[black!60!white] (1.2149999999999999,0.885) rectangle (1.485,0.6149999999999999);
\fill[black!60!white] (1.515,0.885) rectangle (1.785,0.6149999999999999);
\fill[black!60!white] (1.815,0.885) rectangle (2.085,0.6149999999999999);
\fill[black!60!white] (1.2149999999999999,0.585) rectangle (1.485,0.315);
\fill[black!60!white] (1.515,0.585) rectangle (1.785,0.315);
\fill[black!60!white] (1.815,0.585) rectangle (2.085,0.315);
\fill[black!60!white] (0.015,0.285) rectangle (0.285,0.015000000000000013);
\fill[black!60!white] (0.9149999999999999,0.285) rectangle (1.185,0.015000000000000013);
\fill[black!60!white] (1.2149999999999999,0.285) rectangle (1.485,0.015000000000000013);
\fill[black!60!white] (1.815,0.285) rectangle (2.085,0.015000000000000013);
\node[] at (1.05,-0.7) {  $\text{per}\left(A_{7}\right)=64$};\end{tikzpicture}
\hspace{1.0 cm}
\begin{tikzpicture}
\draw[step=0.3cm,gray,thin] (0,0) grid (2.4,2.4);
\fill[black!60!white] (0.015,2.385) rectangle (0.285,2.1149999999999998);
\fill[black!60!white] (0.315,2.385) rectangle (0.585,2.1149999999999998);
\fill[black!60!white] (0.6149999999999999,2.385) rectangle (0.885,2.1149999999999998);
\fill[black!60!white] (0.315,2.085) rectangle (0.585,1.815);
\fill[black!60!white] (0.6149999999999999,2.085) rectangle (0.885,1.815);
\fill[black!60!white] (0.9149999999999999,2.085) rectangle (1.185,1.815);
\fill[black!60!white] (0.315,1.785) rectangle (0.585,1.515);
\fill[black!60!white] (0.6149999999999999,1.785) rectangle (0.885,1.515);
\fill[black!60!white] (0.9149999999999999,1.785) rectangle (1.185,1.515);
\fill[black!60!white] (1.2149999999999999,1.785) rectangle (1.485,1.515);
\fill[black!60!white] (0.015,1.485) rectangle (0.285,1.2149999999999999);
\fill[black!60!white] (0.315,1.485) rectangle (0.585,1.2149999999999999);
\fill[black!60!white] (0.9149999999999999,1.485) rectangle (1.185,1.2149999999999999);
\fill[black!60!white] (1.2149999999999999,1.485) rectangle (1.485,1.2149999999999999);
\fill[black!60!white] (0.015,1.185) rectangle (0.285,0.9149999999999999);
\fill[black!60!white] (1.2149999999999999,1.185) rectangle (1.485,0.9149999999999999);
\fill[black!60!white] (1.515,1.185) rectangle (1.785,0.9149999999999999);
\fill[black!60!white] (1.815,1.185) rectangle (2.085,0.9149999999999999);
\fill[black!60!white] (1.515,0.885) rectangle (1.785,0.6149999999999999);
\fill[black!60!white] (1.815,0.885) rectangle (2.085,0.6149999999999999);
\fill[black!60!white] (2.1149999999999998,0.885) rectangle (2.385,0.6149999999999999);
\fill[black!60!white] (1.515,0.585) rectangle (1.785,0.315);
\fill[black!60!white] (1.815,0.585) rectangle (2.085,0.315);
\fill[black!60!white] (2.1149999999999998,0.585) rectangle (2.385,0.315);
\fill[black!60!white] (0.015,0.285) rectangle (0.285,0.015000000000000013);
\fill[black!60!white] (1.2149999999999999,0.285) rectangle (1.485,0.015000000000000013);
\fill[black!60!white] (1.515,0.285) rectangle (1.785,0.015000000000000013);
\fill[black!60!white] (2.1149999999999998,0.285) rectangle (2.385,0.015000000000000013);
\node[] at (1.2,-0.7) {  $\text{per}\left(A_{8}\right)=120$};\end{tikzpicture}
\hspace{1.0 cm}
\begin{tikzpicture}
\draw[step=0.3cm,gray,thin] (0,0) grid (2.6999999999999997,2.6999999999999997);
\fill[black!60!white] (0.015,2.6849999999999996) rectangle (0.285,2.415);
\fill[black!60!white] (0.315,2.6849999999999996) rectangle (0.585,2.415);
\fill[black!60!white] (0.6149999999999999,2.6849999999999996) rectangle (0.885,2.415);
\fill[black!60!white] (0.315,2.385) rectangle (0.585,2.1149999999999998);
\fill[black!60!white] (0.6149999999999999,2.385) rectangle (0.885,2.1149999999999998);
\fill[black!60!white] (0.9149999999999999,2.385) rectangle (1.185,2.1149999999999998);
\fill[black!60!white] (0.315,2.085) rectangle (0.585,1.815);
\fill[black!60!white] (0.6149999999999999,2.085) rectangle (0.885,1.815);
\fill[black!60!white] (0.9149999999999999,2.085) rectangle (1.185,1.815);
\fill[black!60!white] (1.2149999999999999,2.085) rectangle (1.485,1.815);
\fill[black!60!white] (0.015,1.785) rectangle (0.285,1.515);
\fill[black!60!white] (0.315,1.785) rectangle (0.585,1.515);
\fill[black!60!white] (0.9149999999999999,1.785) rectangle (1.185,1.515);
\fill[black!60!white] (1.2149999999999999,1.785) rectangle (1.485,1.515);
\fill[black!60!white] (0.015,1.485) rectangle (0.285,1.2149999999999999);
\fill[black!60!white] (1.2149999999999999,1.485) rectangle (1.485,1.2149999999999999);
\fill[black!60!white] (1.515,1.485) rectangle (1.785,1.2149999999999999);
\fill[black!60!white] (1.815,1.485) rectangle (2.085,1.2149999999999999);
\fill[black!60!white] (1.515,1.185) rectangle (1.785,0.9149999999999999);
\fill[black!60!white] (1.815,1.185) rectangle (2.085,0.9149999999999999);
\fill[black!60!white] (2.1149999999999998,1.185) rectangle (2.385,0.9149999999999999);
\fill[black!60!white] (1.815,0.885) rectangle (2.085,0.6149999999999999);
\fill[black!60!white] (2.1149999999999998,0.885) rectangle (2.385,0.6149999999999999);
\fill[black!60!white] (2.415,0.885) rectangle (2.6849999999999996,0.6149999999999999);
\fill[black!60!white] (1.515,0.585) rectangle (1.785,0.315);
\fill[black!60!white] (1.815,0.585) rectangle (2.085,0.315);
\fill[black!60!white] (2.1149999999999998,0.585) rectangle (2.385,0.315);
\fill[black!60!white] (2.415,0.585) rectangle (2.6849999999999996,0.315);
\fill[black!60!white] (0.015,0.285) rectangle (0.285,0.015000000000000013);
\fill[black!60!white] (1.2149999999999999,0.285) rectangle (1.485,0.015000000000000013);
\fill[black!60!white] (1.515,0.285) rectangle (1.785,0.015000000000000013);
\fill[black!60!white] (2.415,0.285) rectangle (2.6849999999999996,0.015000000000000013);
\node[] at (1.3499999999999999,-0.7) {  $\text{per}\left(A_{9}\right)=225$};\end{tikzpicture}

\begin{tikzpicture}
\draw[step=0.3cm,gray,thin] (0,0) grid (3.0,3.0);
\fill[black!60!white] (0.015,2.985) rectangle (0.285,2.7150000000000003);
\fill[black!60!white] (0.315,2.985) rectangle (0.585,2.7150000000000003);
\fill[black!60!white] (0.6149999999999999,2.985) rectangle (0.885,2.7150000000000003);
\fill[black!60!white] (0.9149999999999999,2.985) rectangle (1.185,2.7150000000000003);
\fill[black!60!white] (0.6149999999999999,2.6849999999999996) rectangle (0.885,2.415);
\fill[black!60!white] (0.9149999999999999,2.6849999999999996) rectangle (1.185,2.415);
\fill[black!60!white] (1.2149999999999999,2.6849999999999996) rectangle (1.485,2.415);
\fill[black!60!white] (0.9149999999999999,2.385) rectangle (1.185,2.1149999999999998);
\fill[black!60!white] (1.2149999999999999,2.385) rectangle (1.485,2.1149999999999998);
\fill[black!60!white] (1.515,2.385) rectangle (1.785,2.1149999999999998);
\fill[black!60!white] (0.6149999999999999,2.085) rectangle (0.885,1.815);
\fill[black!60!white] (0.9149999999999999,2.085) rectangle (1.185,1.815);
\fill[black!60!white] (1.2149999999999999,2.085) rectangle (1.485,1.815);
\fill[black!60!white] (1.515,2.085) rectangle (1.785,1.815);
\fill[black!60!white] (0.315,1.785) rectangle (0.585,1.515);
\fill[black!60!white] (0.6149999999999999,1.785) rectangle (0.885,1.515);
\fill[black!60!white] (1.515,1.785) rectangle (1.785,1.515);
\fill[black!60!white] (1.815,1.785) rectangle (2.085,1.515);
\fill[black!60!white] (0.015,1.485) rectangle (0.285,1.2149999999999999);
\fill[black!60!white] (0.315,1.485) rectangle (0.585,1.2149999999999999);
\fill[black!60!white] (1.515,1.485) rectangle (1.785,1.2149999999999999);
\fill[black!60!white] (1.815,1.485) rectangle (2.085,1.2149999999999999);
\fill[black!60!white] (0.015,1.185) rectangle (0.285,0.9149999999999999);
\fill[black!60!white] (1.815,1.185) rectangle (2.085,0.9149999999999999);
\fill[black!60!white] (2.1149999999999998,1.185) rectangle (2.385,0.9149999999999999);
\fill[black!60!white] (2.415,1.185) rectangle (2.6849999999999996,0.9149999999999999);
\fill[black!60!white] (2.1149999999999998,0.885) rectangle (2.385,0.6149999999999999);
\fill[black!60!white] (2.415,0.885) rectangle (2.6849999999999996,0.6149999999999999);
\fill[black!60!white] (2.7150000000000003,0.885) rectangle (2.985,0.6149999999999999);
\fill[black!60!white] (2.1149999999999998,0.585) rectangle (2.385,0.315);
\fill[black!60!white] (2.415,0.585) rectangle (2.6849999999999996,0.315);
\fill[black!60!white] (2.7150000000000003,0.585) rectangle (2.985,0.315);
\fill[black!60!white] (0.015,0.285) rectangle (0.285,0.015000000000000013);
\fill[black!60!white] (1.815,0.285) rectangle (2.085,0.015000000000000013);
\fill[black!60!white] (2.1149999999999998,0.285) rectangle (2.385,0.015000000000000013);
\fill[black!60!white] (2.7150000000000003,0.285) rectangle (2.985,0.015000000000000013);
\node[] at (1.5,-0.7) {  $\text{per}\left(A_{10}\right)=424$};\end{tikzpicture}
\hspace{0.2 cm}
\begin{tikzpicture}
\draw[step=0.3cm,gray,thin] (0,0) grid (3.3,3.3);
\fill[black!60!white] (0.015,3.2849999999999997) rectangle (0.285,3.015);
\fill[black!60!white] (0.315,3.2849999999999997) rectangle (0.585,3.015);
\fill[black!60!white] (0.6149999999999999,3.2849999999999997) rectangle (0.885,3.015);
\fill[black!60!white] (0.9149999999999999,3.2849999999999997) rectangle (1.185,3.015);
\fill[black!60!white] (0.6149999999999999,2.985) rectangle (0.885,2.7150000000000003);
\fill[black!60!white] (0.9149999999999999,2.985) rectangle (1.185,2.7150000000000003);
\fill[black!60!white] (1.2149999999999999,2.985) rectangle (1.485,2.7150000000000003);
\fill[black!60!white] (0.9149999999999999,2.6849999999999996) rectangle (1.185,2.415);
\fill[black!60!white] (1.2149999999999999,2.6849999999999996) rectangle (1.485,2.415);
\fill[black!60!white] (1.515,2.6849999999999996) rectangle (1.785,2.415);
\fill[black!60!white] (0.6149999999999999,2.385) rectangle (0.885,2.1149999999999998);
\fill[black!60!white] (0.9149999999999999,2.385) rectangle (1.185,2.1149999999999998);
\fill[black!60!white] (1.2149999999999999,2.385) rectangle (1.485,2.1149999999999998);
\fill[black!60!white] (1.515,2.385) rectangle (1.785,2.1149999999999998);
\fill[black!60!white] (0.315,2.085) rectangle (0.585,1.815);
\fill[black!60!white] (0.6149999999999999,2.085) rectangle (0.885,1.815);
\fill[black!60!white] (1.515,2.085) rectangle (1.785,1.815);
\fill[black!60!white] (1.815,2.085) rectangle (2.085,1.815);
\fill[black!60!white] (0.015,1.785) rectangle (0.285,1.515);
\fill[black!60!white] (0.315,1.785) rectangle (0.585,1.515);
\fill[black!60!white] (1.515,1.785) rectangle (1.785,1.515);
\fill[black!60!white] (1.815,1.785) rectangle (2.085,1.515);
\fill[black!60!white] (0.015,1.485) rectangle (0.285,1.2149999999999999);
\fill[black!60!white] (1.815,1.485) rectangle (2.085,1.2149999999999999);
\fill[black!60!white] (2.1149999999999998,1.485) rectangle (2.385,1.2149999999999999);
\fill[black!60!white] (2.415,1.485) rectangle (2.6849999999999996,1.2149999999999999);
\fill[black!60!white] (2.1149999999999998,1.185) rectangle (2.385,0.9149999999999999);
\fill[black!60!white] (2.415,1.185) rectangle (2.6849999999999996,0.9149999999999999);
\fill[black!60!white] (2.7150000000000003,1.185) rectangle (2.985,0.9149999999999999);
\fill[black!60!white] (2.415,0.885) rectangle (2.6849999999999996,0.6149999999999999);
\fill[black!60!white] (2.7150000000000003,0.885) rectangle (2.985,0.6149999999999999);
\fill[black!60!white] (3.015,0.885) rectangle (3.2849999999999997,0.6149999999999999);
\fill[black!60!white] (2.1149999999999998,0.585) rectangle (2.385,0.315);
\fill[black!60!white] (2.415,0.585) rectangle (2.6849999999999996,0.315);
\fill[black!60!white] (2.7150000000000003,0.585) rectangle (2.985,0.315);
\fill[black!60!white] (3.015,0.585) rectangle (3.2849999999999997,0.315);
\fill[black!60!white] (0.015,0.285) rectangle (0.285,0.015000000000000013);
\fill[black!60!white] (1.815,0.285) rectangle (2.085,0.015000000000000013);
\fill[black!60!white] (2.1149999999999998,0.285) rectangle (2.385,0.015000000000000013);
\fill[black!60!white] (3.015,0.285) rectangle (3.2849999999999997,0.015000000000000013);
\node[] at (1.65,-0.7) {  $\text{per}\left(A_{11}\right)=795$};\end{tikzpicture}
\hspace{0.2 cm}
\begin{tikzpicture}
\draw[step=0.3cm,gray,thin] (0,0) grid (3.5999999999999996,3.5999999999999996);
\fill[black!60!white] (0.015,3.5849999999999995) rectangle (0.285,3.315);
\fill[black!60!white] (0.315,3.5849999999999995) rectangle (0.585,3.315);
\fill[black!60!white] (0.6149999999999999,3.5849999999999995) rectangle (0.885,3.315);
\fill[black!60!white] (0.9149999999999999,3.5849999999999995) rectangle (1.185,3.315);
\fill[black!60!white] (0.6149999999999999,3.2849999999999997) rectangle (0.885,3.015);
\fill[black!60!white] (0.9149999999999999,3.2849999999999997) rectangle (1.185,3.015);
\fill[black!60!white] (1.2149999999999999,3.2849999999999997) rectangle (1.485,3.015);
\fill[black!60!white] (0.6149999999999999,2.985) rectangle (0.885,2.7150000000000003);
\fill[black!60!white] (0.9149999999999999,2.985) rectangle (1.185,2.7150000000000003);
\fill[black!60!white] (1.2149999999999999,2.985) rectangle (1.485,2.7150000000000003);
\fill[black!60!white] (1.515,2.985) rectangle (1.785,2.7150000000000003);
\fill[black!60!white] (0.315,2.6849999999999996) rectangle (0.585,2.415);
\fill[black!60!white] (0.6149999999999999,2.6849999999999996) rectangle (0.885,2.415);
\fill[black!60!white] (1.2149999999999999,2.6849999999999996) rectangle (1.485,2.415);
\fill[black!60!white] (1.515,2.6849999999999996) rectangle (1.785,2.415);
\fill[black!60!white] (0.315,2.385) rectangle (0.585,2.1149999999999998);
\fill[black!60!white] (1.515,2.385) rectangle (1.785,2.1149999999999998);
\fill[black!60!white] (1.815,2.385) rectangle (2.085,2.1149999999999998);
\fill[black!60!white] (0.015,2.085) rectangle (0.285,1.815);
\fill[black!60!white] (0.315,2.085) rectangle (0.585,1.815);
\fill[black!60!white] (1.515,2.085) rectangle (1.785,1.815);
\fill[black!60!white] (1.815,2.085) rectangle (2.085,1.815);
\fill[black!60!white] (0.015,1.785) rectangle (0.285,1.515);
\fill[black!60!white] (1.815,1.785) rectangle (2.085,1.515);
\fill[black!60!white] (2.1149999999999998,1.785) rectangle (2.385,1.515);
\fill[black!60!white] (2.415,1.785) rectangle (2.6849999999999996,1.515);
\fill[black!60!white] (2.1149999999999998,1.485) rectangle (2.385,1.2149999999999999);
\fill[black!60!white] (2.415,1.485) rectangle (2.6849999999999996,1.2149999999999999);
\fill[black!60!white] (2.7150000000000003,1.485) rectangle (2.985,1.2149999999999999);
\fill[black!60!white] (3.015,1.485) rectangle (3.2849999999999997,1.2149999999999999);
\fill[black!60!white] (2.7150000000000003,1.185) rectangle (2.985,0.9149999999999999);
\fill[black!60!white] (3.015,1.185) rectangle (3.2849999999999997,0.9149999999999999);
\fill[black!60!white] (3.315,1.185) rectangle (3.5849999999999995,0.9149999999999999);
\fill[black!60!white] (2.7150000000000003,0.885) rectangle (2.985,0.6149999999999999);
\fill[black!60!white] (3.015,0.885) rectangle (3.2849999999999997,0.6149999999999999);
\fill[black!60!white] (3.315,0.885) rectangle (3.5849999999999995,0.6149999999999999);
\fill[black!60!white] (2.1149999999999998,0.585) rectangle (2.385,0.315);
\fill[black!60!white] (2.415,0.585) rectangle (2.6849999999999996,0.315);
\fill[black!60!white] (2.7150000000000003,0.585) rectangle (2.985,0.315);
\fill[black!60!white] (3.315,0.585) rectangle (3.5849999999999995,0.315);
\fill[black!60!white] (0.015,0.285) rectangle (0.285,0.015000000000000013);
\fill[black!60!white] (1.815,0.285) rectangle (2.085,0.015000000000000013);
\fill[black!60!white] (2.1149999999999998,0.285) rectangle (2.385,0.015000000000000013);
\fill[black!60!white] (3.315,0.285) rectangle (3.5849999999999995,0.015000000000000013);
\node[] at (1.7999999999999998,-0.7) {  $\text{per}\left(A_{12}\right)=1484$};\end{tikzpicture}
\hspace{0.2 cm}
\begin{tikzpicture}
\draw[step=0.3cm,gray,thin] (0,0) grid (3.9,3.9);
\fill[black!60!white] (0.015,3.885) rectangle (0.285,3.615);
\fill[black!60!white] (0.315,3.885) rectangle (0.585,3.615);
\fill[black!60!white] (0.6149999999999999,3.885) rectangle (0.885,3.615);
\fill[black!60!white] (0.9149999999999999,3.885) rectangle (1.185,3.615);
\fill[black!60!white] (0.6149999999999999,3.5849999999999995) rectangle (0.885,3.315);
\fill[black!60!white] (0.9149999999999999,3.5849999999999995) rectangle (1.185,3.315);
\fill[black!60!white] (1.2149999999999999,3.5849999999999995) rectangle (1.485,3.315);
\fill[black!60!white] (0.9149999999999999,3.2849999999999997) rectangle (1.185,3.015);
\fill[black!60!white] (1.2149999999999999,3.2849999999999997) rectangle (1.485,3.015);
\fill[black!60!white] (1.515,3.2849999999999997) rectangle (1.785,3.015);
\fill[black!60!white] (0.6149999999999999,2.985) rectangle (0.885,2.7150000000000003);
\fill[black!60!white] (0.9149999999999999,2.985) rectangle (1.185,2.7150000000000003);
\fill[black!60!white] (1.2149999999999999,2.985) rectangle (1.485,2.7150000000000003);
\fill[black!60!white] (1.515,2.985) rectangle (1.785,2.7150000000000003);
\fill[black!60!white] (0.315,2.6849999999999996) rectangle (0.585,2.415);
\fill[black!60!white] (0.6149999999999999,2.6849999999999996) rectangle (0.885,2.415);
\fill[black!60!white] (1.515,2.6849999999999996) rectangle (1.785,2.415);
\fill[black!60!white] (1.815,2.6849999999999996) rectangle (2.085,2.415);
\fill[black!60!white] (0.015,2.385) rectangle (0.285,2.1149999999999998);
\fill[black!60!white] (0.315,2.385) rectangle (0.585,2.1149999999999998);
\fill[black!60!white] (1.515,2.385) rectangle (1.785,2.1149999999999998);
\fill[black!60!white] (1.815,2.385) rectangle (2.085,2.1149999999999998);
\fill[black!60!white] (0.015,2.085) rectangle (0.285,1.815);
\fill[black!60!white] (1.815,2.085) rectangle (2.085,1.815);
\fill[black!60!white] (2.1149999999999998,2.085) rectangle (2.385,1.815);
\fill[black!60!white] (2.415,2.085) rectangle (2.6849999999999996,1.815);
\fill[black!60!white] (2.1149999999999998,1.785) rectangle (2.385,1.515);
\fill[black!60!white] (2.415,1.785) rectangle (2.6849999999999996,1.515);
\fill[black!60!white] (2.7150000000000003,1.785) rectangle (2.985,1.515);
\fill[black!60!white] (3.015,1.785) rectangle (3.2849999999999997,1.515);
\fill[black!60!white] (2.7150000000000003,1.485) rectangle (2.985,1.2149999999999999);
\fill[black!60!white] (3.015,1.485) rectangle (3.2849999999999997,1.2149999999999999);
\fill[black!60!white] (3.315,1.485) rectangle (3.5849999999999995,1.2149999999999999);
\fill[black!60!white] (2.7150000000000003,1.185) rectangle (2.985,0.9149999999999999);
\fill[black!60!white] (3.015,1.185) rectangle (3.2849999999999997,0.9149999999999999);
\fill[black!60!white] (3.315,1.185) rectangle (3.5849999999999995,0.9149999999999999);
\fill[black!60!white] (3.615,1.185) rectangle (3.885,0.9149999999999999);
\fill[black!60!white] (2.415,0.885) rectangle (2.6849999999999996,0.6149999999999999);
\fill[black!60!white] (2.7150000000000003,0.885) rectangle (2.985,0.6149999999999999);
\fill[black!60!white] (3.315,0.885) rectangle (3.5849999999999995,0.6149999999999999);
\fill[black!60!white] (3.615,0.885) rectangle (3.885,0.6149999999999999);
\fill[black!60!white] (2.1149999999999998,0.585) rectangle (2.385,0.315);
\fill[black!60!white] (2.415,0.585) rectangle (2.6849999999999996,0.315);
\fill[black!60!white] (3.615,0.585) rectangle (3.885,0.315);
\fill[black!60!white] (0.015,0.285) rectangle (0.285,0.015000000000000013);
\fill[black!60!white] (1.815,0.285) rectangle (2.085,0.015000000000000013);
\fill[black!60!white] (2.1149999999999998,0.285) rectangle (2.385,0.015000000000000013);
\fill[black!60!white] (3.615,0.285) rectangle (3.885,0.015000000000000013);
\node[] at (1.95,-0.7) {  $\text{per}\left(A_{13}\right)=2809$};\end{tikzpicture}

\caption{The best constructions we found for Question~\ref{ques:brualdi}, with $\sigma = 312$. Dark squares denote ones and light squares denote zeros.}
    \label{fig:brualdi1}
\end{figure}

It is clear that the matrices in Figure~\ref{fig:brualdi1} have some kind of pattern -- a large main branch going diagonally with some smaller branches coming out of it in the other direction -- but it is not at all clear what this pattern exactly is and what the limiting shape is as $n$ grows to infinity. 

Given two 0-1 square matrices $A$ and $B$, we define $A\ast B$ and $A\circ B$ to be the the matrices obtained from the direct sum of $A$ and $B$, by turning four zeros into ones as in Figure~\ref{fig:operations}. It is easy to verify that if $A$ and $B$ are 312-avoiding, then so are $A\ast B$ and $A\circ B$. Note that it is not the case that every matrix in Figure~\ref{fig:brualdi1} can be obtained by starting from the $1\times 1$ identity matrix, and applying these two operations repeatedly in some order; the first instance of this is at $n=10$, where the largest permanent one can obtain this way has value 394. We can use these operations (or in fact just the matrix direct sum operation) to give a simple lower bound on $f_{312}(n)$:
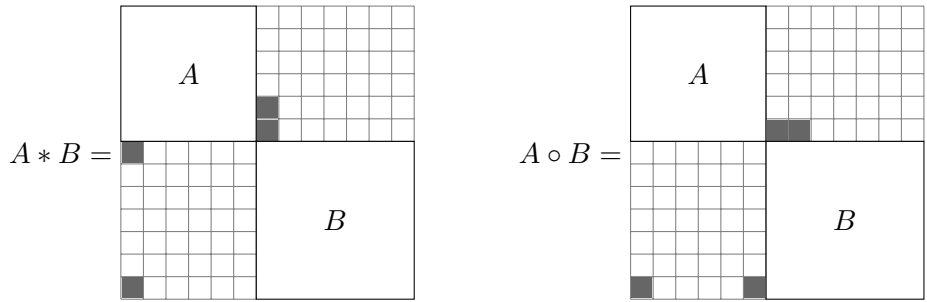
\begin{figure}[hbt]
\centering
\begin{tikzpicture}
\draw[step=0.3cm,gray,thin] (0,0) grid (1.8,2.1);
\draw[step=0.3cm,gray,thin] (1.8,2.1) grid (3.9,3.9);

\fill[black!60!white] (0.015,0.015) rectangle (0.285,0.285);
\fill[black!60!white] (0.015,1.815) rectangle (0.285,2.085);
\fill[black!60!white] (1.815,2.115) rectangle (2.085,2.385);
\fill[black!60!white] (1.815,2.415) rectangle (2.085,2.685);
\draw (0,2.1) rectangle (1.8,3.9);
\draw (1.8,0) rectangle (3.9,2.1);

\node[] at (0.9,3.0) {  $A$};
\node[] at (2.85,1.05) {  $B$};
\node[] at (-0.8,1.95) {  $A\ast B =$};

\end{tikzpicture}
\hspace{1cm}
\begin{tikzpicture}
\draw[step=0.3cm,gray,thin] (0,0) grid (1.8,2.1);
\draw[step=0.3cm,gray,thin] (1.8,2.1) grid (3.9,3.9);

\fill[black!60!white] (0.015,0.015) rectangle (0.285,0.285);
\fill[black!60!white] (1.515,0.015) rectangle (1.785,0.285);
\fill[black!60!white] (1.815,2.115) rectangle (2.085,2.385);
\fill[black!60!white] (2.115,2.115) rectangle (2.385,2.385);
\draw (0,2.1) rectangle (1.8,3.9);
\draw (1.8,0) rectangle (3.9,2.1);

\node[] at (0.9,3.0) {  $A$};
\node[] at (2.85,1.05) {  $B$};
\node[] at (-0.8,1.95) {  $A\circ B =$};

\end{tikzpicture}
\caption{The definition of $A\ast B$ and $A\circ B$. They are both formed from the direct sum of $A$ and $B$, by turning the four zeros indicated by the dark squares into ones. }
\label{fig:operations}
\end{figure}

\begin{proposition}
$$ 2^{0.89n}\leq f_{312}(n)\leq 24^{n/4}\approx 2^{1.15n}$$
\end{proposition}
\begin{proof}
For the upper bound we can use the Br\`egman--Minc theorem~\cite{bregman}, which states that the permanent of an $n\times n$ 0-1 matrix $A=(a_{ij})$ with row sums $r_{i}=a_{i1}+\ldots +a_{in}$ for $i=1,\ldots ,n$ can be estimated by
$$\text{per}(A)\leq \prod_{i=1}^n (r_i!)^{1/r_i}.$$
An $n\times n$ binary matrix without a 312 pattern can have at most $4n-4$ ones, see e.g.~\cite{brualdi}. The above formula is maximized when all $r_i$-s are as equal as possible, this yields the bound in the proposition as $4!=24$.

For the lower bound, observe that for $A_{8}$ as in Figure~\ref{fig:brualdi1}, we have $\text{per}\left(A_{13}\right) = 2809 > 2^{0.88\cdot 13}$. As $\text{per}(A\ast B)\geq \text{per}(A)\cdot \text{per}(B)$ for any $A,B$, this implies that $f_{312}(n+m)\geq f_{312}(n)\cdot f_{312}(m)$ for any integers $m,n\geq 1$. Hence $g(n) := \log_2(f_{312}(n))$ is a superadditive sequence. By Fekete's lemma~\cite{fekete} the limit $\lim_{n\rightarrow\infty}\frac{g(n)}{n}$ exists and is equal to $\sup \frac{g(n)}{n}$ which is at least $\frac{g(13)}{13}\geq 0.88$. 
The better bound of $0.89$ comes from the calculation $\text{per}(A_{13}\circ A_{11})= 5113196>2^{0.89\cdot 25}$.
\end{proof}

It would be interesting to determine the constant $c$ such that $f_{312}(n)=c^{n(1+o(1))}$, and the shape of the optimal constructions as $n\rightarrow\infty$.

\section{Finding constructions with LP solvers}\label{sec:lp}

Many problems in extremal combinatorics can be phrased as linear programs. When this is the case, LP solvers can outperform more general machine learning algorithms, hence LP solvers can be extremely useful for quickly checking conjectures for the existence of small counterexamples. In this section we will present two counterexamples we have obtained this way. The first one is an interesting covering problem in the hypercube, where the authors made the very natural conjecture that embedding the problem in more dimensions cannot make it easier. The second one concerns a conjecture that is closely related to the classical set-pair system inequality of Bollob\'as.

Throughout this section we will use the methods of~\cite{mypaper} to set up the problems as linear programs, and the LP solver Gurobi~\cite{gurobi}.

\subsection{An exact covering problem in the hypercube}\label{subsec:carla}

A vector $a\in \mathbb{R}^n$ and a scalar $b\in \mathbb{R}$ determine the hyperplane $$\{x\in \mathbb{R}^n : \langle a,x \rangle = a_1x_1 + \ldots + a_nx_n = b  \}.$$
A classical result of Alon and F\"uredi~\cite{alon1993covering} states that the number of hyperplanes required to cover precisely $2^{n}-1$ vertices of the hypercube $\{0,1\}^n$, without covering the last vertex, is $n$. This is tight: we can cover $\{0,1\}^n\setminus \textbf{0}$, where $\textbf{0}$ denotes the all zero vector, with the $n$ hyperplanes $\{x:x_i = 1\}$, where $i\in\{1,2,\ldots, n\}$. Several variations of this result have been studied in the literature, see e.g.~the three papers~\cite{aaronson,bishnoi,sauermann} from this year.

For a set $B\subset \{0,1\}^n$ denote by $\text{ec}(B)$ the \emph{exact cover number} of $B$, i.e.~the minimum number of hyperplanes whose union intersects $\{0,1\}^n$ in precisely $B$. Note that the result of Alon and F\"uredi says that $\text{ec}(\{0,1\}^n\setminus \textbf{0})=n$.

Aaronson, Groenland, Grzesik, Kielak, and Johnston~\cite{aaronson} raised the following question. Given a set $B\subset \{0,1\}^k$ and integer $n\geq k$, observe that we have the inequality
\begin{equation*}
    \text{ec}(\{0,1\}^{n} \setminus (B\times \{0\}^{n-k})) \leq n-k + \text{ec}(\{0,1\}^k\setminus B).
\end{equation*}
Indeed, we can cover $\{0,1\}^{n} \setminus (B\times \{0\}^{n-k})$ by using the hyperplanes used to cover $\{0,1\}^k\setminus B$, together with the $n-k$ hyperplanes $\{x:x_i=1\}$, where $i=k+1,\ldots,n$. Intuitively, it seems plausible that this inequality is in fact always sharp: it is not clear how it could help to  place the ‘same subset’ in a space with more dimensions. 
\begin{conjecture}[\label{conj:carla}Aaronson--Groenland--Grzesik--Kielak--Johnston~\cite{aaronson}]
For any $B\subset \{0,1\}^k$ and $n\in\mathbb{N}$ with $n\geq k$, we have
$$\text{ec}(\{0,1\}^{n} \setminus (B\times \{0\}^{n-k})) = n-k + \text{ec}(\{0,1\}^k\setminus B).$$
\end{conjecture}

Once we have fixed a set $B$ and integers $n$ and $k$, finding the relevant exact cover numbers can be phrased as an integer program, with indicator variables for all possible intersections of $\{0,1\}^n$ with a hyperplane. By sampling the set $B$ according to some ad hoc heuristics and solving the resulting linear programs, we were eventually able to find the following counterexample to Conjecture~\ref{conj:carla}.

Let $n=6, k=4$, and   $B = \{1000, 1111, 1001, 1011, 0110, 0001, 0010, 0111\}$. It is straightforward to check via a case analysis that $\{0,1\}^4 \setminus B$ cannot be covered with two hyperplanes, so that $\text{ec}(\{0,1\}^4\setminus B)\geq 3$. On the other hand, surprisingly we can cover $\{0, 1\}^6 \setminus (B \times \{0\}^2)$ with four hyperplanes:
\begin{equation*}
\begin{split}
        \qquad - x_2 ~ \qquad ~ ~ ~  \qquad  +x_5       +x_6 &=1\\
        \qquad     x_2 - ~ x_3 ~ ~ ~  \qquad \qquad      +x_6 &=1\\
        x_1  + x_2 \qquad \quad     - ~ x_4 +x_5     +x_6 &=2\\
        2x_1 - x_2  -2x_3 +2x_4+x_5     -x_6 &=0\\
\end{split}
\end{equation*}
See Figure~\ref{fig:carla} for an illustration of this example.

\begin{figure}[!ht]
    \centering
    \input{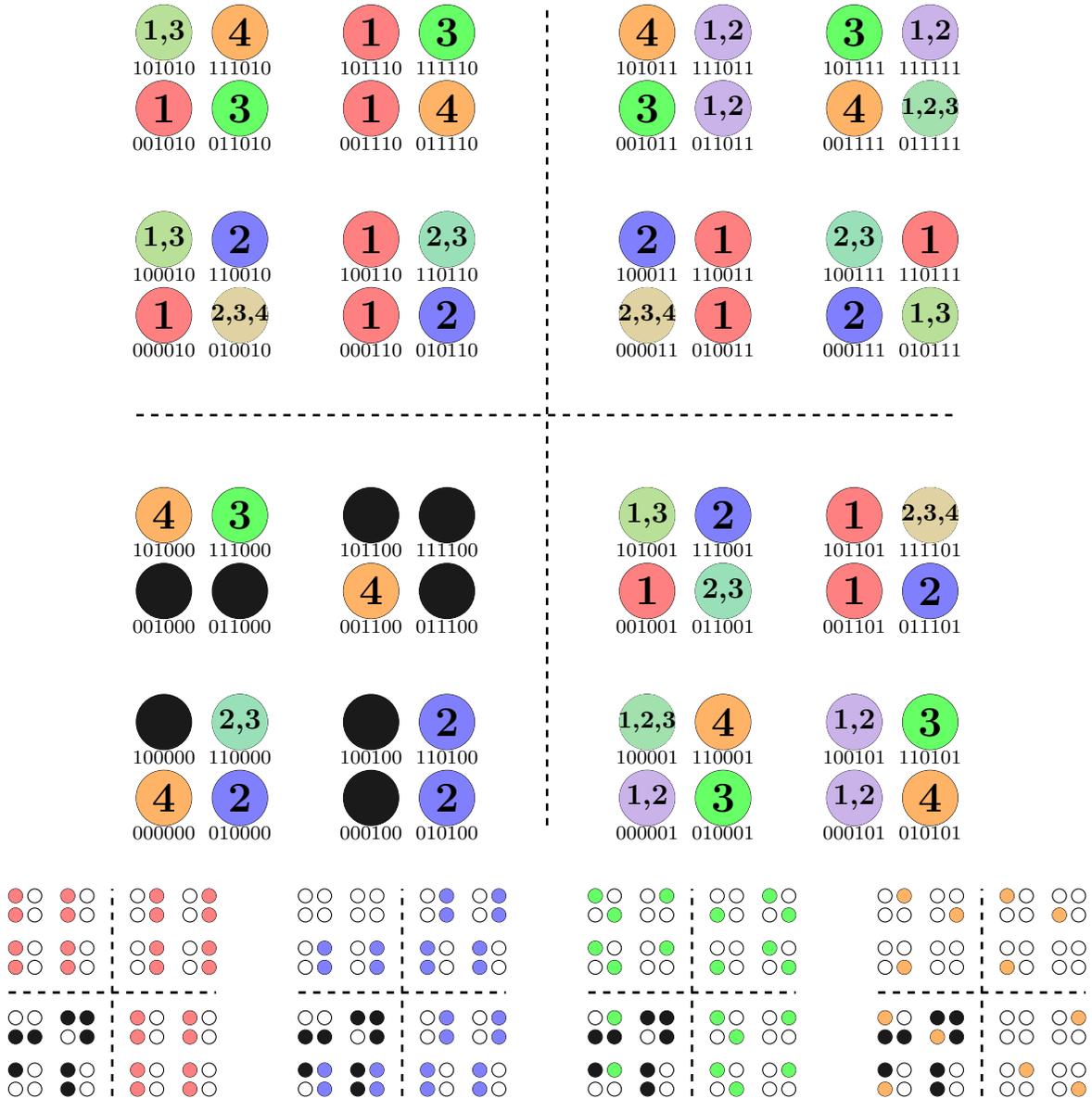}
    
    ~

    \begin{tikzpicture}
\draw (0.0,0.0) circle (0.1cm);
\draw (1.75,0.0) circle (0.1cm);
\draw (0.0,1.75) circle (0.1cm);
\draw (1.75,1.75) circle (0.1cm);
\draw (0.75,0.0) circle (0.1cm);
\draw (2.5,0.0) circle (0.1cm);
\draw (0.75,1.75) circle (0.1cm);
\draw (2.5,1.75) circle (0.1cm);
\draw (0.0,0.75) circle (0.1cm);
\draw (1.75,0.75) circle (0.1cm);
\draw (0.0,2.5) circle (0.1cm);
\draw (1.75,2.5) circle (0.1cm);
\draw (0.75,0.75) circle (0.1cm);
\draw (2.5,0.75) circle (0.1cm);
\draw (0.75,2.5) circle (0.1cm);
\draw (2.5,2.5) circle (0.1cm);
\draw (0.275,0.0) circle (0.1cm);
\draw (2.025,0.0) circle (0.1cm);
\draw (0.275,1.75) circle (0.1cm);
\draw (2.025,1.75) circle (0.1cm);
\draw (1.025,0.0) circle (0.1cm);
\draw (2.775,0.0) circle (0.1cm);
\draw (1.025,1.75) circle (0.1cm);
\draw (2.775,1.75) circle (0.1cm);
\draw (0.275,0.75) circle (0.1cm);
\draw (2.025,0.75) circle (0.1cm);
\draw (0.275,2.5) circle (0.1cm);
\draw (2.025,2.5) circle (0.1cm);
\draw (1.025,0.75) circle (0.1cm);
\draw (2.775,0.75) circle (0.1cm);
\draw (1.025,2.5) circle (0.1cm);
\draw (2.775,2.5) circle (0.1cm);
\draw (0.0,0.275) circle (0.1cm);
\draw (1.75,0.275) circle (0.1cm);
\draw (0.0,2.025) circle (0.1cm);
\draw (1.75,2.025) circle (0.1cm);
\draw (0.75,0.275) circle (0.1cm);
\draw (2.5,0.275) circle (0.1cm);
\draw (0.75,2.025) circle (0.1cm);
\draw (2.5,2.025) circle (0.1cm);
\draw (0.0,1.025) circle (0.1cm);
\draw (1.75,1.025) circle (0.1cm);
\draw (0.0,2.775) circle (0.1cm);
\draw (1.75,2.775) circle (0.1cm);
\draw (0.75,1.025) circle (0.1cm);
\draw (2.5,1.025) circle (0.1cm);
\draw (0.75,2.775) circle (0.1cm);
\draw (2.5,2.775) circle (0.1cm);
\draw (0.275,0.275) circle (0.1cm);
\draw (2.025,0.275) circle (0.1cm);
\draw (0.275,2.025) circle (0.1cm);
\draw (2.025,2.025) circle (0.1cm);
\draw (1.025,0.275) circle (0.1cm);
\draw (2.775,0.275) circle (0.1cm);
\draw (1.025,2.025) circle (0.1cm);
\draw (2.775,2.025) circle (0.1cm);
\draw (0.275,1.025) circle (0.1cm);
\draw (2.025,1.025) circle (0.1cm);
\draw (0.275,2.775) circle (0.1cm);
\draw (2.025,2.775) circle (0.1cm);
\draw (1.025,1.025) circle (0.1cm);
\draw (2.775,1.025) circle (0.1cm);
\draw (1.025,2.775) circle (0.1cm);
\draw (2.775,2.775) circle (0.1cm);
\fill[black!90!white] (0.0,0.275) circle (0.1cm);
\fill[black!90!white] (1.025,1.025) circle (0.1cm);
\fill[black!90!white] (0.75,0.275) circle (0.1cm);
\fill[black!90!white] (0.75,1.025) circle (0.1cm);
\fill[black!90!white] (0.275,0.75) circle (0.1cm);
\fill[black!90!white] (0.75,0.0) circle (0.1cm);
\fill[black!90!white] (0.0,0.75) circle (0.1cm);
\fill[black!90!white] (1.025,0.75) circle (0.1cm);
\fill[red!50!white] (1.75,0.0) circle (0.1cm);
\fill[red!50!white] (0.0,1.75) circle (0.1cm);
\fill[red!50!white] (2.5,0.0) circle (0.1cm);
\fill[red!50!white] (0.75,1.75) circle (0.1cm);
\fill[red!50!white] (1.75,0.75) circle (0.1cm);
\fill[red!50!white] (0.0,2.5) circle (0.1cm);
\fill[red!50!white] (2.5,0.75) circle (0.1cm);
\fill[red!50!white] (0.75,2.5) circle (0.1cm);
\fill[red!50!white] (2.025,1.75) circle (0.1cm);
\fill[red!50!white] (2.775,1.75) circle (0.1cm);
\fill[red!50!white] (2.025,2.5) circle (0.1cm);
\fill[red!50!white] (2.775,2.5) circle (0.1cm);
\fill[red!50!white] (1.75,0.275) circle (0.1cm);
\fill[red!50!white] (0.0,2.025) circle (0.1cm);
\fill[red!50!white] (2.5,0.275) circle (0.1cm);
\fill[red!50!white] (0.75,2.025) circle (0.1cm);
\fill[red!50!white] (1.75,1.025) circle (0.1cm);
\fill[red!50!white] (0.0,2.775) circle (0.1cm);
\fill[red!50!white] (2.5,1.025) circle (0.1cm);
\fill[red!50!white] (0.75,2.775) circle (0.1cm);
\fill[red!50!white] (2.025,2.025) circle (0.1cm);
\fill[red!50!white] (2.775,2.025) circle (0.1cm);
\fill[red!50!white] (2.025,2.775) circle (0.1cm);
\fill[red!50!white] (2.775,2.775) circle (0.1cm);
\draw [line width=0.35mm, dashed] (1.3875,-0.1) -- (1.3875,2.875);
\draw [line width=0.35mm, dashed] (-0.1,1.3875) -- (2.875,1.3875);
\end{tikzpicture}
    \hspace{0.3in}
    \begin{tikzpicture}
\draw (0.0,0.0) circle (0.1cm);
\draw (1.75,0.0) circle (0.1cm);
\draw (0.0,1.75) circle (0.1cm);
\draw (1.75,1.75) circle (0.1cm);
\draw (0.75,0.0) circle (0.1cm);
\draw (2.5,0.0) circle (0.1cm);
\draw (0.75,1.75) circle (0.1cm);
\draw (2.5,1.75) circle (0.1cm);
\draw (0.0,0.75) circle (0.1cm);
\draw (1.75,0.75) circle (0.1cm);
\draw (0.0,2.5) circle (0.1cm);
\draw (1.75,2.5) circle (0.1cm);
\draw (0.75,0.75) circle (0.1cm);
\draw (2.5,0.75) circle (0.1cm);
\draw (0.75,2.5) circle (0.1cm);
\draw (2.5,2.5) circle (0.1cm);
\draw (0.275,0.0) circle (0.1cm);
\draw (2.025,0.0) circle (0.1cm);
\draw (0.275,1.75) circle (0.1cm);
\draw (2.025,1.75) circle (0.1cm);
\draw (1.025,0.0) circle (0.1cm);
\draw (2.775,0.0) circle (0.1cm);
\draw (1.025,1.75) circle (0.1cm);
\draw (2.775,1.75) circle (0.1cm);
\draw (0.275,0.75) circle (0.1cm);
\draw (2.025,0.75) circle (0.1cm);
\draw (0.275,2.5) circle (0.1cm);
\draw (2.025,2.5) circle (0.1cm);
\draw (1.025,0.75) circle (0.1cm);
\draw (2.775,0.75) circle (0.1cm);
\draw (1.025,2.5) circle (0.1cm);
\draw (2.775,2.5) circle (0.1cm);
\draw (0.0,0.275) circle (0.1cm);
\draw (1.75,0.275) circle (0.1cm);
\draw (0.0,2.025) circle (0.1cm);
\draw (1.75,2.025) circle (0.1cm);
\draw (0.75,0.275) circle (0.1cm);
\draw (2.5,0.275) circle (0.1cm);
\draw (0.75,2.025) circle (0.1cm);
\draw (2.5,2.025) circle (0.1cm);
\draw (0.0,1.025) circle (0.1cm);
\draw (1.75,1.025) circle (0.1cm);
\draw (0.0,2.775) circle (0.1cm);
\draw (1.75,2.775) circle (0.1cm);
\draw (0.75,1.025) circle (0.1cm);
\draw (2.5,1.025) circle (0.1cm);
\draw (0.75,2.775) circle (0.1cm);
\draw (2.5,2.775) circle (0.1cm);
\draw (0.275,0.275) circle (0.1cm);
\draw (2.025,0.275) circle (0.1cm);
\draw (0.275,2.025) circle (0.1cm);
\draw (2.025,2.025) circle (0.1cm);
\draw (1.025,0.275) circle (0.1cm);
\draw (2.775,0.275) circle (0.1cm);
\draw (1.025,2.025) circle (0.1cm);
\draw (2.775,2.025) circle (0.1cm);
\draw (0.275,1.025) circle (0.1cm);
\draw (2.025,1.025) circle (0.1cm);
\draw (0.275,2.775) circle (0.1cm);
\draw (2.025,2.775) circle (0.1cm);
\draw (1.025,1.025) circle (0.1cm);
\draw (2.775,1.025) circle (0.1cm);
\draw (1.025,2.775) circle (0.1cm);
\draw (2.775,2.775) circle (0.1cm);
\fill[black!90!white] (0.0,0.275) circle (0.1cm);
\fill[black!90!white] (1.025,1.025) circle (0.1cm);
\fill[black!90!white] (0.75,0.275) circle (0.1cm);
\fill[black!90!white] (0.75,1.025) circle (0.1cm);
\fill[black!90!white] (0.275,0.75) circle (0.1cm);
\fill[black!90!white] (0.75,0.0) circle (0.1cm);
\fill[black!90!white] (0.0,0.75) circle (0.1cm);
\fill[black!90!white] (1.025,0.75) circle (0.1cm);
\fill[blue!50!white] (1.75,0.0) circle (0.1cm);
\fill[blue!50!white] (1.75,1.75) circle (0.1cm);
\fill[blue!50!white] (2.5,0.0) circle (0.1cm);
\fill[blue!50!white] (2.5,1.75) circle (0.1cm);
\fill[blue!50!white] (0.275,0.0) circle (0.1cm);
\fill[blue!50!white] (0.275,1.75) circle (0.1cm);
\fill[blue!50!white] (1.025,0.0) circle (0.1cm);
\fill[blue!50!white] (1.025,1.75) circle (0.1cm);
\fill[blue!50!white] (2.025,0.75) circle (0.1cm);
\fill[blue!50!white] (2.025,2.5) circle (0.1cm);
\fill[blue!50!white] (2.775,0.75) circle (0.1cm);
\fill[blue!50!white] (2.775,2.5) circle (0.1cm);
\fill[blue!50!white] (1.75,0.275) circle (0.1cm);
\fill[blue!50!white] (1.75,2.025) circle (0.1cm);
\fill[blue!50!white] (2.5,0.275) circle (0.1cm);
\fill[blue!50!white] (2.5,2.025) circle (0.1cm);
\fill[blue!50!white] (0.275,0.275) circle (0.1cm);
\fill[blue!50!white] (0.275,2.025) circle (0.1cm);
\fill[blue!50!white] (1.025,0.275) circle (0.1cm);
\fill[blue!50!white] (1.025,2.025) circle (0.1cm);
\fill[blue!50!white] (2.025,1.025) circle (0.1cm);
\fill[blue!50!white] (2.025,2.775) circle (0.1cm);
\fill[blue!50!white] (2.775,1.025) circle (0.1cm);
\fill[blue!50!white] (2.775,2.775) circle (0.1cm);
\draw [line width=0.35mm, dashed] (1.3875,-0.1) -- (1.3875,2.875);
\draw [line width=0.35mm, dashed] (-0.1,1.3875) -- (2.875,1.3875);
\end{tikzpicture}
    \hspace{0.3in}
    \begin{tikzpicture}
\draw (0.0,0.0) circle (0.1cm);
\draw (1.75,0.0) circle (0.1cm);
\draw (0.0,1.75) circle (0.1cm);
\draw (1.75,1.75) circle (0.1cm);
\draw (0.75,0.0) circle (0.1cm);
\draw (2.5,0.0) circle (0.1cm);
\draw (0.75,1.75) circle (0.1cm);
\draw (2.5,1.75) circle (0.1cm);
\draw (0.0,0.75) circle (0.1cm);
\draw (1.75,0.75) circle (0.1cm);
\draw (0.0,2.5) circle (0.1cm);
\draw (1.75,2.5) circle (0.1cm);
\draw (0.75,0.75) circle (0.1cm);
\draw (2.5,0.75) circle (0.1cm);
\draw (0.75,2.5) circle (0.1cm);
\draw (2.5,2.5) circle (0.1cm);
\draw (0.275,0.0) circle (0.1cm);
\draw (2.025,0.0) circle (0.1cm);
\draw (0.275,1.75) circle (0.1cm);
\draw (2.025,1.75) circle (0.1cm);
\draw (1.025,0.0) circle (0.1cm);
\draw (2.775,0.0) circle (0.1cm);
\draw (1.025,1.75) circle (0.1cm);
\draw (2.775,1.75) circle (0.1cm);
\draw (0.275,0.75) circle (0.1cm);
\draw (2.025,0.75) circle (0.1cm);
\draw (0.275,2.5) circle (0.1cm);
\draw (2.025,2.5) circle (0.1cm);
\draw (1.025,0.75) circle (0.1cm);
\draw (2.775,0.75) circle (0.1cm);
\draw (1.025,2.5) circle (0.1cm);
\draw (2.775,2.5) circle (0.1cm);
\draw (0.0,0.275) circle (0.1cm);
\draw (1.75,0.275) circle (0.1cm);
\draw (0.0,2.025) circle (0.1cm);
\draw (1.75,2.025) circle (0.1cm);
\draw (0.75,0.275) circle (0.1cm);
\draw (2.5,0.275) circle (0.1cm);
\draw (0.75,2.025) circle (0.1cm);
\draw (2.5,2.025) circle (0.1cm);
\draw (0.0,1.025) circle (0.1cm);
\draw (1.75,1.025) circle (0.1cm);
\draw (0.0,2.775) circle (0.1cm);
\draw (1.75,2.775) circle (0.1cm);
\draw (0.75,1.025) circle (0.1cm);
\draw (2.5,1.025) circle (0.1cm);
\draw (0.75,2.775) circle (0.1cm);
\draw (2.5,2.775) circle (0.1cm);
\draw (0.275,0.275) circle (0.1cm);
\draw (2.025,0.275) circle (0.1cm);
\draw (0.275,2.025) circle (0.1cm);
\draw (2.025,2.025) circle (0.1cm);
\draw (1.025,0.275) circle (0.1cm);
\draw (2.775,0.275) circle (0.1cm);
\draw (1.025,2.025) circle (0.1cm);
\draw (2.775,2.025) circle (0.1cm);
\draw (0.275,1.025) circle (0.1cm);
\draw (2.025,1.025) circle (0.1cm);
\draw (0.275,2.775) circle (0.1cm);
\draw (2.025,2.775) circle (0.1cm);
\draw (1.025,1.025) circle (0.1cm);
\draw (2.775,1.025) circle (0.1cm);
\draw (1.025,2.775) circle (0.1cm);
\draw (2.775,2.775) circle (0.1cm);
\fill[black!90!white] (0.0,0.275) circle (0.1cm);
\fill[black!90!white] (1.025,1.025) circle (0.1cm);
\fill[black!90!white] (0.75,0.275) circle (0.1cm);
\fill[black!90!white] (0.75,1.025) circle (0.1cm);
\fill[black!90!white] (0.275,0.75) circle (0.1cm);
\fill[black!90!white] (0.75,0.0) circle (0.1cm);
\fill[black!90!white] (0.0,0.75) circle (0.1cm);
\fill[black!90!white] (1.025,0.75) circle (0.1cm);
\fill[green!60!white] (1.75,1.75) circle (0.1cm);
\fill[green!60!white] (1.75,2.5) circle (0.1cm);
\fill[green!60!white] (2.025,0.0) circle (0.1cm);
\fill[green!60!white] (0.275,1.75) circle (0.1cm);
\fill[green!60!white] (2.775,1.75) circle (0.1cm);
\fill[green!60!white] (2.025,0.75) circle (0.1cm);
\fill[green!60!white] (0.275,2.5) circle (0.1cm);
\fill[green!60!white] (2.775,2.5) circle (0.1cm);
\fill[green!60!white] (1.75,0.275) circle (0.1cm);
\fill[green!60!white] (0.0,2.025) circle (0.1cm);
\fill[green!60!white] (2.5,2.025) circle (0.1cm);
\fill[green!60!white] (1.75,1.025) circle (0.1cm);
\fill[green!60!white] (0.0,2.775) circle (0.1cm);
\fill[green!60!white] (2.5,2.775) circle (0.1cm);
\fill[green!60!white] (0.275,0.275) circle (0.1cm);
\fill[green!60!white] (2.775,0.275) circle (0.1cm);
\fill[green!60!white] (1.025,2.025) circle (0.1cm);
\fill[green!60!white] (0.275,1.025) circle (0.1cm);
\fill[green!60!white] (2.775,1.025) circle (0.1cm);
\fill[green!60!white] (1.025,2.775) circle (0.1cm);
\draw [line width=0.35mm, dashed] (1.3875,-0.1) -- (1.3875,2.875);
\draw [line width=0.35mm, dashed] (-0.1,1.3875) -- (2.875,1.3875);
\end{tikzpicture}
    \hspace{0.3in}
    \begin{tikzpicture}
\draw (0.0,0.0) circle (0.1cm);
\draw (1.75,0.0) circle (0.1cm);
\draw (0.0,1.75) circle (0.1cm);
\draw (1.75,1.75) circle (0.1cm);
\draw (0.75,0.0) circle (0.1cm);
\draw (2.5,0.0) circle (0.1cm);
\draw (0.75,1.75) circle (0.1cm);
\draw (2.5,1.75) circle (0.1cm);
\draw (0.0,0.75) circle (0.1cm);
\draw (1.75,0.75) circle (0.1cm);
\draw (0.0,2.5) circle (0.1cm);
\draw (1.75,2.5) circle (0.1cm);
\draw (0.75,0.75) circle (0.1cm);
\draw (2.5,0.75) circle (0.1cm);
\draw (0.75,2.5) circle (0.1cm);
\draw (2.5,2.5) circle (0.1cm);
\draw (0.275,0.0) circle (0.1cm);
\draw (2.025,0.0) circle (0.1cm);
\draw (0.275,1.75) circle (0.1cm);
\draw (2.025,1.75) circle (0.1cm);
\draw (1.025,0.0) circle (0.1cm);
\draw (2.775,0.0) circle (0.1cm);
\draw (1.025,1.75) circle (0.1cm);
\draw (2.775,1.75) circle (0.1cm);
\draw (0.275,0.75) circle (0.1cm);
\draw (2.025,0.75) circle (0.1cm);
\draw (0.275,2.5) circle (0.1cm);
\draw (2.025,2.5) circle (0.1cm);
\draw (1.025,0.75) circle (0.1cm);
\draw (2.775,0.75) circle (0.1cm);
\draw (1.025,2.5) circle (0.1cm);
\draw (2.775,2.5) circle (0.1cm);
\draw (0.0,0.275) circle (0.1cm);
\draw (1.75,0.275) circle (0.1cm);
\draw (0.0,2.025) circle (0.1cm);
\draw (1.75,2.025) circle (0.1cm);
\draw (0.75,0.275) circle (0.1cm);
\draw (2.5,0.275) circle (0.1cm);
\draw (0.75,2.025) circle (0.1cm);
\draw (2.5,2.025) circle (0.1cm);
\draw (0.0,1.025) circle (0.1cm);
\draw (1.75,1.025) circle (0.1cm);
\draw (0.0,2.775) circle (0.1cm);
\draw (1.75,2.775) circle (0.1cm);
\draw (0.75,1.025) circle (0.1cm);
\draw (2.5,1.025) circle (0.1cm);
\draw (0.75,2.775) circle (0.1cm);
\draw (2.5,2.775) circle (0.1cm);
\draw (0.275,0.275) circle (0.1cm);
\draw (2.025,0.275) circle (0.1cm);
\draw (0.275,2.025) circle (0.1cm);
\draw (2.025,2.025) circle (0.1cm);
\draw (1.025,0.275) circle (0.1cm);
\draw (2.775,0.275) circle (0.1cm);
\draw (1.025,2.025) circle (0.1cm);
\draw (2.775,2.025) circle (0.1cm);
\draw (0.275,1.025) circle (0.1cm);
\draw (2.025,1.025) circle (0.1cm);
\draw (0.275,2.775) circle (0.1cm);
\draw (2.025,2.775) circle (0.1cm);
\draw (1.025,1.025) circle (0.1cm);
\draw (2.775,1.025) circle (0.1cm);
\draw (1.025,2.775) circle (0.1cm);
\draw (2.775,2.775) circle (0.1cm);
\fill[black!90!white] (0.0,0.275) circle (0.1cm);
\fill[black!90!white] (1.025,1.025) circle (0.1cm);
\fill[black!90!white] (0.75,0.275) circle (0.1cm);
\fill[black!90!white] (0.75,1.025) circle (0.1cm);
\fill[black!90!white] (0.275,0.75) circle (0.1cm);
\fill[black!90!white] (0.75,0.0) circle (0.1cm);
\fill[black!90!white] (0.0,0.75) circle (0.1cm);
\fill[black!90!white] (1.025,0.75) circle (0.1cm);
\fill[orange!60!white] (0.0,0.0) circle (0.1cm);
\fill[orange!60!white] (1.75,1.75) circle (0.1cm);
\fill[orange!60!white] (0.75,0.75) circle (0.1cm);
\fill[orange!60!white] (2.5,2.5) circle (0.1cm);
\fill[orange!60!white] (0.275,1.75) circle (0.1cm);
\fill[orange!60!white] (2.775,0.0) circle (0.1cm);
\fill[orange!60!white] (1.025,2.5) circle (0.1cm);
\fill[orange!60!white] (0.0,1.025) circle (0.1cm);
\fill[orange!60!white] (1.75,2.775) circle (0.1cm);
\fill[orange!60!white] (2.025,0.275) circle (0.1cm);
\fill[orange!60!white] (0.275,2.775) circle (0.1cm);
\fill[orange!60!white] (2.775,1.025) circle (0.1cm);
\draw [line width=0.35mm, dashed] (1.3875,-0.1) -- (1.3875,2.875);
\draw [line width=0.35mm, dashed] (-0.1,1.3875) -- (2.875,1.3875);
\end{tikzpicture}
    \caption{Covering $\{0,1\}^6\setminus (B\times \{0\}^2)$ with four hyperplanes. The circles represent the points of $\{0,1\}^6$. Black disks are elements of $B$, otherwise the numbers in a circle represent which of the four hyperplanes contain that point.}
    \label{fig:carla}
\end{figure}

\subsection{Weakly cross-intersecting \texorpdfstring{$(a,b)$}{(a,b)}-set systems}\label{subsec:domotor}

Let $\F$ be a family of pairs subsets $(A_1,B_1), (A_2,B_2), \ldots$ of $\mathbb{N}$. We say that $\F$ is an $(a,b)$\emph{-set system} if for every set-pair $(A_i,B_i)$ in $\F$, we have that $|A_i|=a$, $|B_i|=b$, and $A_i\cap B_i=\emptyset$. We say that $\F$ has the \emph{cross-intersecting property} if for any two pairs $(A_i,B_i)$ and $(A_j,B_j)$ in $\F$ with $i\neq j$, we have that $A_i\cap B_j\neq \emptyset$ and $A_j\cap B_i \neq \emptyset$. A classical result of Bollob\'as~\cite{bollsetpair} states that if $\F$ is an $(a,b)$-set system with the cross-intersecting property, then $|\F|\leq \binom{a+b}{b}$, independently of the size of the ground set. As shown by Frankl~\cite{frankl}, the same conclusion also holds if we relax the cross-intersecting property, and only require that $A_i\cap B_j\neq \emptyset$ when $i<j$. 

Kir\'aly, Nagy, P\'alv\"olgyi and Visontai~\cite{domotor} considered what happens when we further relax the cross-intersecting condition. They called an $(a,b)$-set system \emph{weakly cross-intersecting} if for any $i\neq j$, we have that at most one of the two sets $A_i\cap B_j$ and $A_j\cap B_i$ is empty. They denoted by $g(a,b)$ the maximum size of a weakly cross-intersecting $(a,b)$-set system. They showed that if $a+b\rightarrow\infty$, we have $g(a,b)\geq (2-o(1))\binom{a+b}{b}$. Among others, they raised the following problem:

\begin{problem}[Kir\'aly--Nagy--P\'alv\"olgyi--Visontai~\cite{domotor}\label{prob:domotor}]
Is $g(a,b) < 2\binom{a+b}{b}$?
\end{problem}
We have found a $(4,4)$-set system of size $146>140=2\cdot\binom{8}{4}$, which shows that the answer to Problem~\ref{prob:domotor} is in general ``no''. It is given in the appendix. It would be interesting to see how one can generalize this construction, and what the right order of magnitude of $g(a,b)$ is. A result of Tuza~\cite{tuza} implies $g(a,a)\leq 2^{2a}$, but even the answer to the following problem is not known.

\begin{problem}[Kir\'aly--Nagy--P\'alv\"olgyi--Visontai~\cite{domotor}]
Is $g(a,a)=o\left(2^{2a}\right)$?
\end{problem}

\section{Concluding remarks}\label{sec:concl}

The main contribution of the present work is that we have demonstrated some  success with applying reinforcement learning methods to find explicit constructions and counterexamples to problems in combinatorics. All examples presented in Section~\ref{sec:neural} used the cross-entropy method. According to~\cite{lapan2020deep}, the main advantages of the cross-entropy method are that it is a very simple algorithm that has good convergence and works well in simple environments that do not require us to learn complex, multistep policies. This makes it an ideal baseline method to try.

While the cross-entropy method works well in general, there exist a plethora of more sophisticated reinforcement learning algorithms that could potentially perform much better for some problems. It would be extremely interesting to see some success with refuting conjectures in combinatorics, graph theory, or other areas of mathematics, by finding explicit counterexamples using other reinforcement learning algorithms.

\begin{problem}
Use a different reinforcement learning algorithm to find an explicit counterexample to an open conjecture in mathematics.
\end{problem}

~

\textbf{Acknowledgment:} The author is very grateful to B\'alint Varga for help with various programming-related questions.

\vspace{-0.2in}

\bibliography{mybib}
\bibliographystyle{abbrv}

\appendix
\section{The construction for Problem~\ref{prob:domotor}}

The $(4,4)$-set system in $\{1, 2,\ldots, 11\}$ of size 146 is as follows.

~

\small{(1, 2, 3, 7 | 5, 6, 9, 10), \quad (1, 2, 3, 9 | 5, 8, 10, 11), \quad (1, 2, 4, 6 | 5, 8, 9, 11), \quad (1, 2, 4, 7 | 3, 5, 6, 11), \quad \\(1, 2, 4, 8 | 5, 7, 9, 11), \quad (1, 2, 4, 11 | 3, 5, 8, 9), \quad (1, 2, 5, 6 | 7, 9, 10, 11), \quad (1, 2, 5, 7 | 6, 8, 9, 11), \quad \\(1, 2, 5, 9 | 4, 6, 8, 11), \quad (1, 2, 5, 10 | 6, 7, 8, 9), \quad (1, 2, 6, 7 | 4, 9, 10, 11), \quad (1, 2, 6, 10 | 4, 8, 9, 11), \quad \\(1, 2, 7, 10 | 4, 5, 6, 8), \quad (1, 2, 8, 11 | 3, 5, 6, 10), \quad (1, 2, 9, 11 | 3, 4, 8, 10), \quad (1, 3, 4, 8 | 2, 5, 6, 11), \quad \\(1, 3, 4, 9 | 2, 5, 6, 8), \quad (1, 3, 4, 10 | 5, 6, 8, 9), \quad (1, 3, 5, 8 | 6, 9, 10, 11), \quad (1, 3, 5, 9 | 2, 6, 7, 10), \quad \\(1, 3, 5, 10 | 2, 6, 7, 11), \quad (1, 3, 6, 8 | 4, 5, 7, 9), \quad (1, 3, 6, 10 | 2, 4, 8, 9), \quad (1, 3, 7, 10 | 2, 4, 6, 9), \quad \\(1, 3, 7, 11 | 2, 5, 9, 10), \quad (1, 3, 8, 11 | 5, 6, 7, 10), \quad (1, 4, 5, 6 | 2, 9, 10, 11), \quad (1, 4, 5, 8 | 3, 6, 7, 11), \quad \\(1, 4, 5, 11 | 6, 8, 9, 10), \quad (1, 4, 6, 11 | 2, 5, 7, 8), \quad (1, 4, 7, 8 | 2, 3, 6, 9), \quad (1, 4, 7, 9 | 2, 3, 5, 11), \quad \\(1, 4, 8, 9 | 3, 5, 7, 11), \quad (1, 4, 9, 10 | 3, 6, 7, 8), \quad (1, 5, 6, 9 | 7, 8, 10, 11), \quad (1, 5, 6, 11 | 3, 4, 7, 9), \quad \\(1, 5, 7, 9 | 2, 6, 10, 11), \quad (1, 5, 7, 11 | 2, 4, 6, 10), \quad (1, 5, 8, 9 | 3, 4, 6, 7), \quad (1, 5, 8, 11 | 6, 7, 9, 10), \quad \\(1, 5, 10, 11 | 2, 6, 7, 9), \quad (1, 6, 7, 8 | 2, 5, 9, 11), \quad (1, 6, 7, 9 | 2, 4, 5, 11), \quad (1, 6, 7, 11 | 2, 3, 9, 10), \quad \\(1, 6, 10, 11 | 2, 3, 4, 5), \quad (1, 7, 8, 9 | 4, 5, 6, 11), \quad (1, 7, 9, 10 | 2, 4, 6, 8), \quad (1, 7, 9, 11 | 2, 3, 5, 10), \quad \\(1, 8, 9, 10 | 4, 5, 7, 11), \quad (1, 8, 10, 11 | 5, 6, 7, 9), \quad (1, 9, 10, 11 | 4, 5, 6, 7), \quad (2, 3, 4, 6 | 1, 5, 10, 11), \quad \\(2, 3, 4, 7 | 1, 5, 6, 10), \quad (2, 3, 5, 7 | 1, 4, 10, 11), \quad (2, 3, 5, 10 | 1, 7, 8, 11), \quad (2, 3, 5, 11 | 1, 6, 7, 10), \quad \\(2, 3, 6, 11 | 5, 8, 9, 10), \quad (2, 3, 7, 9 | 1, 4, 5, 10), \quad (2, 3, 7, 10 | 1, 5, 6, 11), \quad (2, 3, 8, 9 | 6, 7, 10, 11), \quad \\(2, 3, 8, 10 | 1, 5, 7, 9), \quad (2, 3, 8, 11 | 1, 5, 7, 10), \quad (2, 3, 9, 10 | 5, 7, 8, 11), \quad (2, 3, 9, 11 | 5, 7, 8, 10), \quad \\(2, 3, 10, 11 | 1, 6, 8, 9), \quad (2, 4, 5, 7 | 1, 8, 9, 10), \quad (2, 4, 5, 8 | 1, 3, 7, 10), \quad (2, 4, 5, 9 | 6, 8, 10, 11), \quad \\(2, 4, 6, 9 | 3, 5, 7, 8), \quad (2, 4, 6, 10 | 1, 5, 8, 9), \quad (2, 4, 7, 8 | 1, 3, 10, 11), \quad (2, 4, 8, 9 | 1, 3, 5, 7), \quad \\(2, 4, 8, 10 | 1, 3, 9, 11), \quad (2, 4, 9, 11 | 3, 6, 8, 10), \quad (2, 4, 10, 11 | 1, 3, 6, 9), \quad (2, 5, 6, 9 | 1, 7, 8, 10), \quad \\(2, 5, 6, 11 | 1, 4, 7, 9), \quad (2, 5, 7, 8 | 3, 4, 6, 10), \quad (2, 5, 7, 10 | 1, 8, 9, 11), \quad (2, 5, 7, 11 | 4, 8, 9, 10), \quad \\(2, 5, 8, 10 | 4, 6, 7, 9), \quad (2, 6, 7, 9 | 3, 8, 10, 11), \quad (2, 6, 7, 11 | 3, 4, 5, 9), \quad (2, 6, 8, 9 | 3, 4, 10, 11), \quad \\(2, 6, 8, 11 | 3, 5, 7, 9), \quad (2, 6, 9, 10 | 3, 4, 8, 11), \quad (2, 6, 10, 11 | 4, 5, 7, 8), \quad (2, 7, 8, 10 | 3, 4, 9, 11), \quad \\(2, 7, 9, 10 | 1, 3, 6, 11), \quad (2, 8, 9, 11 | 1, 3, 4, 7), \quad (2, 9, 10, 11 | 1, 5, 6, 8), \quad (3, 4, 5, 8 | 1, 7, 9, 11), \quad \\(3, 4, 5, 9 | 1, 2, 7, 11), \quad (3, 4, 5, 11 | 1, 2, 6, 10), \quad (3, 4, 6, 8 | 2, 5, 7, 11), \quad (3, 4, 6, 9 | 2, 5, 8, 11), \quad \\(3, 4, 6, 10 | 2, 8, 9, 11), \quad (3, 4, 7, 8 | 1, 2, 9, 11), \quad (3, 4, 7, 11 | 1, 2, 5, 9), \quad (3, 4, 9, 11 | 1, 2, 5, 10), \quad \\(3, 5, 6, 7 | 2, 4, 8, 10), \quad (3, 5, 6, 8 | 2, 4, 9, 10), \quad (3, 5, 7, 10 | 1, 2, 4, 11), \quad (3, 5, 10, 11 | 1, 2, 8, 9), \quad \\(3, 6, 7, 8 | 1, 4, 5, 9), \quad (3, 6, 7, 10 | 1, 4, 5, 8), \quad (3, 6, 8, 9 | 4, 7, 10, 11), \quad (3, 6, 9, 11 | 2, 4, 5, 7), \quad \\(3, 7, 8, 9 | 1, 2, 10, 11), \quad (3, 7, 8, 11 | 1, 4, 6, 9), \quad (3, 7, 9, 10 | 1, 2, 5, 6), \quad (3, 7, 9, 11 | 2, 4, 5, 10), \quad \\(3, 8, 9, 10 | 1, 5, 6, 7), \quad (3, 8, 10, 11 | 1, 2, 7, 9), \quad (4, 5, 6, 7 | 1, 2, 8, 10), \quad (4, 5, 6, 10 | 3, 7, 8, 9), \quad \\(4, 5, 6, 11 | 7, 8, 9, 10), \quad (4, 5, 7, 8 | 1, 2, 3, 11), \quad (4, 5, 7, 10 | 2, 6, 8, 11), \quad (4, 5, 9, 10 | 1, 3, 7, 11), \quad \\(4, 5, 9, 11 | 2, 3, 6, 10), \quad (4, 6, 7, 10 | 2, 3, 8, 9), \quad (4, 6, 8, 10 | 2, 3, 7, 9), \quad (4, 6, 8, 11 | 2, 7, 9, 10), \quad \\(4, 7, 8, 10 | 1, 2, 3, 5), \quad (4, 7, 8, 11 | 1, 3, 9, 10), \quad (4, 7, 10, 11 | 2, 3, 6, 8), \quad (4, 8, 9, 11 | 2, 3, 5, 7), \quad \\(4, 9, 10, 11 | 1, 2, 7, 8), \quad (5, 6, 7, 9 | 2, 3, 4, 10), \quad (5, 6, 7, 10 | 2, 3, 4, 8), \quad (5, 6, 8, 9 | 2, 3, 7, 10), \quad \\(5, 6, 8, 10 | 4, 7, 9, 11), \quad (5, 6, 9, 10 | 2, 4, 7, 8), \quad (5, 6, 9, 11 | 2, 7, 8, 10), \quad (5, 7, 8, 10 | 2, 3, 4, 11), \quad \\(5, 7, 9, 11 | 1, 4, 6, 8), \quad (5, 7, 10, 11 | 3, 4, 6, 9), \quad (5, 8, 9, 10 | 1, 4, 7, 11), \quad (5, 8, 9, 11 | 1, 2, 4, 6), \quad \\(5, 9, 10, 11 | 4, 6, 7, 8), \quad (6, 7, 9, 11 | 1, 3, 5, 8), \quad (6, 8, 9, 10 | 1, 4, 5, 11), \quad (6, 8, 10, 11 | 1, 2, 3, 4), \quad \\(7, 8, 9, 11 | 1, 3, 5, 10), \quad (7, 8, 10, 11 | 3, 4, 5, 6)
}
\end{document}